\documentclass[12pt,lenq]{amsart}      
\usepackage{graphicx}
\usepackage{amsmath}
\usepackage{amscd}
\usepackage{amssymb}
\usepackage{amsbsy}
\usepackage{amsfonts}
\usepackage{latexsym}
\usepackage{graphics}
\usepackage{amsmath,amscd,latexsym}
\usepackage{multirow}
\usepackage{array}
\usepackage{paralist}
\usepackage{titletoc}
\theoremstyle{plain}
\newtheorem{theorem}{Theorem}[section]
\newtheorem{proposition}[theorem]{Proposition}
\newtheorem{lemma}[theorem]{Lemma}
\newtheorem{corollary}[theorem]{Corollary}
\newtheorem{remark}[theorem]{Remark}
\newtheorem{definition}[theorem]{Definition}

\newtheorem{example}[theorem]{Example}
\newtheorem{main theorem}[theorem]{Main Theorem}

\newtheorem{question}[theorem]{Question}
\newtheorem{convention}[theorem]{Convention}

\newenvironment{rcase}{\left.\begin{aligned}}
{\end{aligned}\right\rbrace}

\newcommand{\ZZ}{\mathbb{Z}}
\newcommand{\QQ}{\mathbb{Q}}
\newcommand{\RR}{\mathbb{R}}

\newcommand{\HH}{\mathbb{H}}
\newcommand{\QQQ}{\hat{\mathbb{Q}}}
\newcommand{\RRR}{\hat{\mathbb{R}}}

\newcommand{\Conway}{\mbox{\boldmath$S$}^{2}}
\newcommand{\Conways}
{(\mbox{\boldmath$S$}^{2},\mbox{\boldmath$P$})}
\newcommand{\PP}{\mbox{\boldmath$P$}}
\newcommand{\PConway}{\mbox{\boldmath$S$}}

\newcommand{\rtangle}[1]{(B^3,t({#1}))}

\newcommand{\DD}{\mathcal{D}}
\newcommand{\RGPP}[1]{\hat\Gamma_{#1}}
\newcommand{\RGP}[1]{\Gamma_{#1}}

\newcommand{\llangle}{\langle\langle}
\newcommand{\rrangle}{\rangle\rangle}

\newcommand{\lp}{(\hskip -0.07cm (}
\newcommand{\rp}{)\hskip -0.07cm )}

\begin{document}

\title{Homotopically equivalent simple loops
on 2-bridge spheres in 2-bridge link complements (I)}
\author{Donghi Lee}
\address{Department of Mathematics\\
Pusan National University \\
San-30 Jangjeon-Dong, Geumjung-Gu, Pusan, 609-735, Republic of Korea}
\email{donghi@pusan.ac.kr}

\author{Makoto Sakuma}
\address{Department of Mathematics\\
Graduate School of Science\\
Hiroshima University\\
Higashi-Hiroshima, 739-8526, Japan}
\email{sakuma@math.sci.hiroshima-u.ac.jp}

\subjclass[2010]{Primary 20F06, 57M25 \\
\indent {The first author was supported by Basic Science Research Program
through the National Research Foundation of Korea(NRF) funded
by the Ministry of Education, Science and Technology(2009-0065798).
The second author was supported
by JSPS Grants-in-Aid 22340013 and 21654011.}}

\begin{abstract}
In this paper and its two
sequels, we give a necessary and sufficient condition
for two essential simple loops
on a 2-bridge sphere in a 2-bridge link complement
to be homotopic in the link complement.
This paper treats the case when the 2-bridge link is a $(2,p)$-torus link,
where more cases of homotopy arise,
and its sequels will treat the remaining cases.
\end{abstract}
\maketitle

\section{Introduction}
Let $K$ be a $2$-bridge link in $S^3$ and let $S$
be a $4$-punctured sphere
in $S^3-K$ obtained from a $2$-bridge sphere of $K$.
In \cite{lee_sakuma}, we gave a complete characterization of
those essential simple loops in $S$
which are null-homotopic in $S^3-K$,
and by using the result,
we described all upper-meridian-pair-preserving epimorphisms between $2$-bridge link groups.
The purpose of this paper and its sequels is to
give a necessary and sufficient condition
for two essential simple loops on $S$
to be homotopic in $S^3-K$.

It has been proved
by Weinbaum~\cite{Weinbaum} and
Appel and Schupp~\cite{Appel-Schupp}
that the word and conjugacy problems
for prime alternating link groups are solvable,
by using small cancellation theory
(see also \cite{Johnsgard} and references in it).
Moreover, it was also shown
by Sela~\cite{Sela} and Pr\'eaux~\cite{Preaux} that
the word and conjugacy problems
for any link group are solvable.
A characteristic feature of
this series of papers including \cite{lee_sakuma}
is that we give a complete answer to
special (but also natural)
word and conjugacy problems
for the groups of 2-bridge links,
which form a special (but also important) family
of prime alternating links.
The key tool used in the proofs is small cancellation theory,
applied to two-generator and one-relator presentations
of $2$-bridge link groups.

In this paper, we treat the case
when the $2$-bridge link is a $(2,p)$-torus link
(Main Theorem~\ref{main_theorem}),
where more cases of homotopy arise.
The remaining cases will be treated in the sequels of this paper
\cite{lee_sakuma_3} and \cite{lee_sakuma_4},
and these results will be used in \cite{lee_sakuma_5}
to show the existence of a variation of McShane's identity for $2$-bridge links.
For an overview of this series of works, we refer the reader to
the research announcement~\cite{lee_sakuma_6}.

This paper is organized as follows.
In Section~\ref{statements},
we recall basic facts concerning $2$-bridge links,
and describe the main result of this paper.
In Section~\ref{preliminaries},
we recall the upper presentation of a $2$-bridge link group,
and recall key facts established in \cite{lee_sakuma}
concerning the upper presentation.
These facts are to be used throughout this series of papers.
In Section~\ref{annular_diagrams},
we establish a very strong
structure theorem (Theorems~\ref{structure}
and \ref{cor:structure})
for the annular diagram
which arises in the study of the conjugacy problem
by using small cancellation theory.
Finally, Section~\ref{sec:only_if_part} is
devoted to the proof of the main result.

\section{Main result}
\label{statements}

Consider the discrete group, $H$, of isometries
of the Euclidean plane $\RR^2$
generated by the $\pi$-rotations around
the points in the lattice $\ZZ^2$.
Set $\Conways:=(\RR^2,\ZZ^2)/H$
and call it the {\it Conway sphere}.
Then $\Conway$ is homeomorphic to the 2-sphere,
and $\PP$ consists of four points in $\Conway$.
We also call $\Conway$ the Conway sphere.
Let $\PConway:=\Conway-\PP$ be the complementary
4-times punctured sphere.
For each $r \in \QQQ:=\QQ\cup\{\infty\}$,
let $\alpha_r$ be the unoriented simple loop in $\PConway$
obtained as the projection of any straight
line in $\RR^2-\ZZ^2$ of slope $r$.
Then $\alpha_r$ is {\it essential} in $\PConway$,
i.e., it does not bound a disk
nor a once-punctured disk in $\PConway$.
Conversely, any essential simple loop in $\PConway$
is isotopic to $\alpha_r$ for
a unique $r\in\QQQ$.
Then $r$ is called the {\it slope} of the simple loop.
Similarly, any simple arc $\delta$
in $\Conway$ joining two
different points in $\PP$ such that
$\delta\cap \PP=\partial\delta$
is isotopic to the image of a line in $\RR^2$
of some slope $r\in\QQQ$
which intersects $\ZZ^2$.
We call $r$ the {\it slope} of $\delta$.
Thus, for every slope $r\in\QQQ$, there exist two arcs and one loop of slope $r$
in $\Conways$ (all unoriented).

A {\it trivial tangle} is a pair $(B^3,t)$,
where $B^3$ is a 3-ball and $t$ is a union of two
arcs properly embedded in $B^3$
which is parallel to a union of two
mutually disjoint arcs in $\partial B^3$.
By a {\it rational tangle},
we mean a trivial tangle $(B^3,t)$
which is endowed with a homeomorphism from
$(\partial B^3,\partial t)$
to $\Conways$.
Through the homeomorphism we identify
the boundary of a rational tangle with the Conway sphere.
Thus the slope of an essential simple loop in
$\partial B^3-t$ is defined.
We define
the {\it slope} of a rational tangle
to be the slope of
an essential loop on $\partial B^3 -t$ which bounds a disk in $B^3$
separating the components of $t$.
(Such a loop is unique up to isotopy
on $\partial B^3 -t$ and is called a {\it meridian}
of the rational tangle.)
We denote a rational tangle of slope $r$ by
$\rtangle{r}$.
By van Kampen's theorem, the fundamental group
$\pi_1(B^3-t(r))$ is identified
with the quotient
$\pi_1(\PConway)/\llangle\alpha_r \rrangle$,
where $\llangle\alpha_r \rrangle$ denotes the normal
closure.

For each $r\in \QQQ$,
the {\it 2-bridge link $K(r)$ of slope $r$}
is defined to be the sum of the rational tangles of slopes
$\infty$ and $r$, namely,
$(S^3,K(r))$ is
obtained from $\rtangle{\infty}$ and
$\rtangle{r}$
by identifying their boundaries through the
identity map on the Conway sphere
$\Conways$. (Recall that the boundaries of
rational tangles are identified with the Conway sphere.)
By van Kampen's theorem again,
the {\it link group} $G(K(r)):=\pi_1(S^3-K(r))$ is identified with
$\pi_1(\PConway)/ \llangle\alpha_{\infty},\alpha_r\rrangle$.
Note that $K(r)$ has one or two components according as
the denominator of $r$ is odd or even.
We call $\rtangle{\infty}$
and $\rtangle{r}$, respectively,
the {\it upper tangle} and {\it lower tangle}
of the 2-bridge link.
Also note that $(S^3, K(r))$ is homeomorphic to $(S^3, K(r+1))$,
so that $r$ matters only modulo $1$.

Let $\DD$ be the
{\it Farey tessellation},
that is,
the tessellation of the upper half
plane $\HH^2$ by ideal triangles which are obtained
from the ideal triangle with the ideal vertices $0, 1,
\infty \in \QQQ$ by repeated reflection in the edges.
Then $\QQQ$ is identified with the set of the ideal vertices of $\DD$.
For each $r\in \QQQ$,
let $\RGP{r}$ be the group of automorphisms of
$\DD$ generated by reflections in the edges of $\DD$
with an endpoint $r$.
It should be noted that $\RGP{r}$
is isomorphic to the infinite dihedral group
and the region bounded by two adjacent edges of $\DD$
with an endpoint $r$ is a fundamental domain
for the action of $\RGP{r}$ on $\HH^2$,
by virtue of Poincar\'{e}'s
fundamental polyhedron theorem
(see, for example, \cite{Ratcliffe}).
Let $\RGPP{r}$ be the group generated by $\RGP{r}$ and $\RGP{\infty}$.
When $r\in \QQ - \ZZ$,
$\RGPP{r}$ is equal to the free product $\RGP{r}*\RGP{\infty}$,
having a fundamental domain shown in Figure~\ref{fig.fd}.
Otherwise, $\RGPP{r}$ is the group generated by the reflections
in the edges of $\DD$
or $\RGP{\infty}$ according as $r\in\ZZ$ or $r=\infty$.

\begin{figure}[h]
\includegraphics{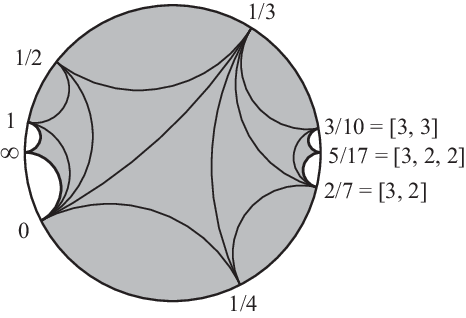}
\caption{
A fundamental domain of $\hat\Gamma_r$ in the
Farey tessellation (the shaded domain) for $r=5/17=${\scriptsize $\cfrac{1}{3+\cfrac{1}{2+\cfrac{1}{2}}}$}\,$=:[3,2,2]$}
\label{fig.fd}
\end{figure}

The following key observation was made
in \cite{Ohtsuki-Riley-Sakuma}.

\begin{proposition} {\rm \cite[Proposition~4.6]{Ohtsuki-Riley-Sakuma}}
\label{ORS-observation}
Let $r \in \QQQ$.
If two elements $s$ and $s'$ of $\QQQ$ belong to the same $\RGPP{r}$-orbit,
then the unoriented loops $\alpha_s$ and $\alpha_{s'}$ are homotopic in $S^3-K(r)$.
\end{proposition}

\begin{proof}
We give a proof to this proposition from a view point brought to the authors by the referee.
Let $\gamma_0\in\Gamma_{\infty}$ be the reflection of $\DD$
in the Farey edge $\langle \infty,0\rangle$.
Note that $\gamma_0$ is induced by a self-homeomorphism, $\tilde\gamma_0$, of $\PConway$
which is a reflection in the simple loop $\alpha_0$ of slope $0$.
Recall that $\PConway$ is regarded as a subspace of $S^3-K(r)$,
and denote by $j:\PConway\to S^3-K(r)$ the inclusion map.
Then $j$ can be homotoped to $j\circ\tilde\gamma_0$
by swapping the upper and lower hemispheres of $\PConway$ bounded $\alpha_0$
through the tangle $(B^3,t(\infty))$, as in Figure~\ref{fig-homotopy}(a).
Hence $\alpha_{\gamma_0(s)}=j\circ\tilde\gamma_0(\alpha_s)$
is homotopic to $j(\alpha_s)=\alpha_s$ in $S^3-K(r)$.

Let $\gamma_1\in\Gamma_{\infty}$ be the reflection of $\DD$
in the Farey edge $\langle \infty,1\rangle$,
and set $\varphi=\gamma_1\gamma_0\in\Gamma_{\infty}$.
Then $\varphi$ is the parabolic transformation of $\DD$,
centered on $\infty$, by two units in the counter-clockwise direction,
and it is induced by a self-homeomorphism $\tilde\varphi$ of $\PConway$
which is a Dehn twist along the simple loop $\alpha_{\infty}$.
Since $\alpha_{\infty}$ bounds a disk in $B^3-t(\infty)$,
$j$ can be homotoped to $j\circ\tilde\varphi$ by a homotopy in $B^3-t(\infty)$.
Hence $\varphi$ also preserves the homotopy classes,
in $S^3-K(r)$, of simple loops in $\PConway$.

Since $\Gamma_{\infty}$ is generated by $\gamma_0$ and $\varphi$,
the above observations imply that $\Gamma_{\infty}$ preserves the homotopy classes,
in $S^3-K(r)$, of simple loops in $\PConway$.
We can also see that $\Gamma_{r}$ has the same property
by using homotopies in $B^3-t(r)$.
Hence we obtain the desired result.
\end{proof}

Thus the following question naturally arises.

\begin{question} \label{question}
Is the converse to the above proposition valid?
If not, when are two essential simple loops on
a 2-bridge sphere in a 2-bridge link complement
homotopic in the link complement?
\end{question}

If $r=\infty$, then $G(K(\infty))$ is a rank $2$ free group,
and the result of Komori and Series \cite[Theorem~1.2]{Komori-Series}
is equivalent to the affirmative answer to the first question in the above.
On the other hand, if $r=0$ (or any integer), then $G(K(0))$ is
the infinite cyclic group,
and $\RGPP{0}$ is equal to the group
generated by the reflections in the edges of any of $\DD$.
In particular, any Farey triangle
is a fundamental domain for the action of $\RGPP{0}$ on $\HH^2$.
Hence, any $s\in\QQQ$ belongs to the $\RGPP{0}$-orbit of one
and only one of $\{0,1,\infty\}$.
But since $\alpha_1$ is not null-homotopic while $\alpha_0$ and $\alpha_{\infty}$
are null-homotopic in $S^3-K(0)$, every simple loop which is not null-homotopic
has slope belonging to the $\RGPP{0}$-orbit of $1$.
This yields the affirmative answer to the first question in the above for $r=0$.
In \cite{lee_sakuma},
we gave the following complete characterization of those
essential simple loops on $2$-bridge spheres of $2$-bridge links
which are null-homotopic in the link complements.

\begin{theorem}{\rm \cite[Main Theorem~2.3]{lee_sakuma}}
\label{previous_main_theorem}
Let $r \in \QQQ$.
The loop $\alpha_s$ is null-homotopic in $S^3 - K(r)$
if and only if $s$ belongs to the $\RGPP{r}$-orbit of $\infty$ or $r$.
\end{theorem}

The purpose of this paper and its sequels is to solve
the above question.
In this paper,
we give an answer to the question
for the $(2,p)$-torus link $K(1/p)$ with $p \ge 2$
(Main Theorem~\ref{main_theorem}).

In order to state the main result,
we introduce some notation.
For a rational number $r=1/p$,
let $\tau$ be the automorphism of the Farey tessellation
defined as follows:
\begin{enumerate} [\indent \rm (i)]
\item If $p=2p_0$ for some $p_0\in \ZZ_+$,
then $\tau$ is the reflection in the Farey edge
$\langle 0, 1/p_0\rangle$.

\item If $p=2p_0+1$ for some $p_0\in \ZZ_+$,
then $\tau$ is the reflection in the geodesic
with an endpoint $0$ which bisects the Farey edge
$\langle 1/p_0, 1/(p_0+1)\rangle$.
Thus $\tau$ is the reflection in the geodesic
with endpoints $0$ and $2/(2p_0+1)$.
\end{enumerate}
In both cases, $\tau$ is the reflection in the geodesic
with endpoints $0$ and $2/p$,
and interchanges $\infty$ with $1/p$ (see Figure~\ref{fig.involution}).
Moreover, we can see that the action of $\tau$ on the vertex set
$\QQQ$ of the Farey tessellation is given by
$\tau(c/d)=c/(cp-d)$.
In particular, if $\tau(q_1/p_1)=q_2/p_2$,
where $(p_i, q_i)$ is a pair of
relatively prime positive integers,
then $q_1=q_2$ and $q_1/(p_1+p_2)=1/p$.

\begin{figure}[h]
\includegraphics{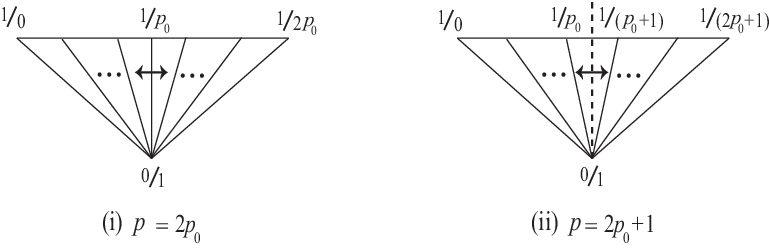}
\caption{\label{fig.involution}
The involution $\tau$}
\end{figure}

Set $\tilde{\Gamma}_{1/p}=\langle \RGPP{1/p}, \tau \rangle$,
the group generated by $\RGPP{1/p}$ and $\tau$.
Then $\tilde{\Gamma}_{1/p}$ is a $\ZZ_2$-extension of $\RGPP{1/p}$.
Now our main theorem is stated as follows.

\begin{main theorem} \label{main_theorem0}
Let $p \ge 2$ be an integer.
Then, for two distinct $s, s' \in \QQQ$,
the unoriented loops $\alpha_s$ and $\alpha_{s'}$ are homotopic in $S^3-K(1/p)$
if and only if
$s$ and $s'$ lie in the same $\tilde{\Gamma}_{1/p}$-orbit.
\end{main theorem}

The if part of Main Theorem~\ref{main_theorem0} can be easily proved by
the following nice observation pointed out by the referee.
By virtue of by Proposition~\ref{ORS-observation},
we have only to show that
$\tau$ preserves the homotopy classes, in $S^3-K(1/p)$,
of simple loops in $\PConway$.
Since the reflection $\gamma_0\in \RGPP{1/p}$
in the Farey edge $\langle \infty, 0 \rangle$
has this property by Proposition~\ref{ORS-observation},
we may prove that $\psi:=\tau\circ\gamma_0$ has this property.
Since $\psi$ is the parabolic transformation of $\DD$,
centered on $0$, by $p$ units in the counter-clockwise direction,
it is induced by a self-homeomorphism $\tilde\psi$ of $\PConway$
which is the $p$ half twists along the simple loop $\alpha_{0}$.
We may assume that the the half twists are made along a parallel copy
of $\alpha_{0}$ lying in the interior of the lower hemisphere
so that the support of $\tilde\psi$ is contained in the lower hemisphere.
Then the inclusion map $j:\PConway\to S^3-K(1/p)$ is homotoped to
$j\circ \tilde\psi$
by pushing the lower-hemisphere sweeping out the inner tangle $(B^3,t(1/p))$,
then through the upper hemisphere back to itself
sweeping out the outer tangle $(B^3,t(\infty))$, as in Figure~\ref{fig-homotopy}(b).
Hence we obtain the the desired result as in the proof of Proposition~\ref{ORS-observation}.

\begin{figure}[h]
\includegraphics{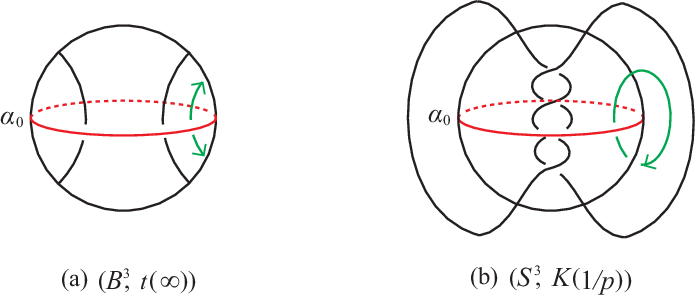}
\caption{\label{fig-homotopy}}
\end{figure}

In the following, we give a reformulation of the main theorem
so that it is more suitable for the proof of the only if part.
Suppose that $r$ is a rational number with $0<r<1$.
Write
\begin{center}
\begin{picture}(230,70)
\put(0,48){$\displaystyle{
r=
\cfrac{1}{m_1+
\cfrac{1}{ \raisebox{-5pt}[0pt][0pt]{$m_2 \, + \, $}
\raisebox{-10pt}[0pt][0pt]{$\, \ddots \ $}
\raisebox{-12pt}[0pt][0pt]{$+ \, \cfrac{1}{m_k}$}
}} \
=:[m_1,m_2, \dots,m_k],}$}
\end{picture}
\end{center}
where $k \ge 1$, $(m_1, \dots, m_k) \in (\mathbb{Z}_+)^k$, and $m_k \ge 2$.
Recall that
the region, $R$, bounded by a pair of
Farey edges with an endpoint $\infty$
and a pair of Farey edges with an endpoint $r$
forms a fundamental domain of the action of $\RGPP{r}$ on $\HH^2$
(see Figure~\ref{fig.fd}).
Let $I_1(r)$ and $I_2(r)$ be the closed intervals in $\RRR$
obtained as the intersection with
$\RRR-\{r, \infty\}$
of the closure of $R$.
To be precise, $I_1(r)=[0,r_1]$ and $I_2(r)=[r_2,1]$,
where
\[
\begin{aligned}
&\begin{rcase}
r_1 &=[m_1, m_2, \dots, m_{k-1}] \\
r_2 &=[m_1, m_2, \dots, m_{k-1}, m_k-1] \,
\end{rcase}
\quad \text{if $k$ is odd,}
\\
&\begin{rcase}
r_1 &=[m_1, m_2, \dots, m_{k-1}, m_k-1] \\
r_2 &=[m_1, m_2, \dots, m_{k-1}] \,
\end{rcase}
\quad \text{if $k$ is even.}
\end{aligned}
\]
If $r=1/p$ ($p \ge 2$), then
$I_1(r)$
is degenerate to the singleton
$\{0\}$.
And if $r=(p-1)/p$ ($p \ge 2$), then
$I_2(r)$ is degenerate to the singleton
$\{1\}$.
Otherwise, $I_1(r)$ and $I_2(r)$ are
non-degenerate intervals,
and the union $I_1(r) \cup I_2(r)$ forms
a fundamental domain of the action of $\RGPP{r}$
on the domain of discontinuity of $\RGPP{r}$,
the complement in $\RRR$ of the closure
of $\RGPP{r}\{\infty, r\}$.
(In the exceptional case $r=1/p$ (resp., $(p-1)/p$),
the rational number
$0$ (resp., $1$)
lies in the limit set and
$I_2(r)$ (resp., $I_1(r)$)
is a fundamental domain of the action of
$\RGPP{r}$ on the domain of discontinuity.)

\begin{lemma}{\rm \cite[Lemma~7.1]{lee_sakuma}}
\label{Lemma:FundametalDomain}
Suppose $0<r<1$.
Then, for any $s\in\QQQ$,
there is a unique rational number
$s_0\in I_1(r) \cup I_2(r) \cup \{\infty, r\}$
such that $s$ is contained in the $\RGPP{r}$-orbit of $s_0$,
and in particular, $\alpha_s$ is homotopic to $\alpha_{s_0}$ in
$S^3-K(r)$.
\end{lemma}

Thus the {\it only if} part of Theorem~\ref{previous_main_theorem}
when $r \notin \ZZ \cup \{\infty\}$
is equivalent to the following theorem.

\begin{theorem}{\rm \cite[Theorem~7.2]{lee_sakuma}}
\label{if_part_theorem}
Suppose $0<r<1$.
Then, for any rational number $s\in I_1(r) \cup I_2(r)$,
$\alpha_s$ is not null-homotopic in $S^3-K(r)$.
\end{theorem}

We can now reformulate our main theorem.

\begin{main theorem} \label{main_theorem}
Let $p \ge 2$ be an integer.
Then, for two distinct rational numbers $s, s' \in I_1(1/p) \cup I_2(1/p)
=\{0\} \cup [{1 \over {p-1}}, 1]$,
the unoriented loops $\alpha_s$ and $\alpha_{s'}$ are homotopic in $S^3-K(1/p)$
if and only if
$s'=\tau(s)$, or in other words
$s=q_1/p_1$ and $s'=q_2/p_2$ satisfy
$q_1=q_2$ and $q_1/(p_1+p_2)=1/p$, where $(p_i, q_i)$ is a pair of
relatively prime positive integers.
\end{main theorem}

We prove the only if part by
interpreting the situation in terms of combinatorial group theory.
In other words, we prove that
two words representing the free homotopy classes of
$\alpha_s$ and $\alpha_{s'}$
are conjugate in the $2$-bridge link group $G(K(1/p))$
only if
$s$ and $s'$ satisfy the conditions given in the statement of the theorem.
The key tool used in the proof is small cancellation theory,
which is one of the representative geometric techniques
in combinatorial group theory, applied to reduced annular diagrams over 2-bridge link groups. The proof is contained in Section~\ref{sec:only_if_part}.

\section{Preliminaries} \label{preliminaries}

In this section, we introduce the
upper presentation of a 2-bridge link group,
and recall key facts established in \cite{lee_sakuma}
concerning it.
These facts are to be used throughout this series of papers.

To find a presentation of
the $2$-bridge link group
$G(K(r))$ explicitly,
let $a$ and $b$, respectively, be the elements of
$\pi_1(B^3-t(\infty), x_0)$
represented by the oriented loops $\mu_1$ and $\mu_2$
based on $x_0$ as illustrated in Figure~\ref{fig.generator}.
Then $\{a,b\}$ forms the meridian pair of
$\pi_1(B^3-t(\infty))$, which is identified with
the free group $F(a,b)$.
Note that
$\mu_i$ intersects the disk, $\delta_i$, in $B^3$
bounded by a component of $t(\infty)$ and
the essential arc, $\gamma_i$, on
$\partial(B^3,t(\infty))=\Conways$ of slope $1/0$,
in Figure~\ref{fig.generator}.
Obtain a word $u_r$ in
$\{a, b\}$ by reading the intersection of the
(suitably oriented) loop $\alpha_r$
with $\gamma_1\cup \gamma_2$,
where a positive intersection with $\gamma_1$ (resp., $\gamma_2$) corresponds to $a$ (resp., $b$).
Then the cyclic word $(u_r)$
represents the free homotopy class of the (oriented) loop $\alpha_r$
(see the paragraph preceding Definition~\ref{def:alternating}
for the precise definition of a cyclic word).
It then follows that
\[
\begin{aligned}
G(K(r))&=\pi_1(S^3-K(r))\cong\pi_1(B^3-t(\infty))/\llangle \alpha_r\rrangle \\
&\cong F(a, b)/ \llangle u_r \rrangle
\cong \langle a, b \, | \, u_r \rangle.
\end{aligned}
\]
This one-relator presentation is
called the {\it upper presentation} of $G(K(r))$ (see \cite{Crowell-Fox}).
If $r \neq \infty$,
then $\alpha_r$ intersects $\gamma_1$ and $\gamma_2$ alternately,
and hence $a$ and $b$ appear in $(u_r)$ alternately
(with exponents $\pm 1$).
It is known by \cite[Proposition~1]{Riley}
that there is a nice formula to find $u_r$ as follows
(see \cite[Remark~1]{lee_sakuma}
for a geometric picture).

\begin{figure}[h]
\includegraphics{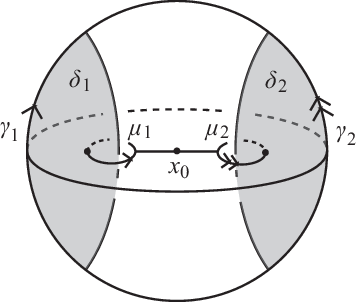}
\caption{
$\pi_1(B^3-t(\infty), x_0)=F(a,b)$,
where $a$ and $b$ are represented by $\mu_1$ and $\mu_2$, respectively.
}
\label{fig.generator}
\end{figure}

\begin{lemma}
\label{presentation}
Let $p$ and $q$ be relatively prime positive integers
such that $p \ge 1$.
For $1 \le i \le p-1$, let
\[\epsilon_i = (-1)^{\lfloor iq/p \rfloor},\]
where $\lfloor x \rfloor$ is the greatest integer not exceeding $x$.
\begin{enumerate}[\indent \rm (1)]
\item If $p$ is odd, then
\[
u_{q/p}=a\hat{u}_{q/p}b^{(-1)^q}\hat{u}_{q/p}^{-1},\]
where
$\hat{u}_{q/p} = b^{\epsilon_1} a^{\epsilon_2} \cdots b^{\epsilon_{p-2}} a^{\epsilon_{p-1}}$.
\item If $p$ is even, then
\[
u_{q/p}=a\hat{u}_{q/p}a^{-1}\hat{u}_{q/p}^{-1},\]
where
$\hat{u}_{q/p} = b^{\epsilon_1} a^{\epsilon_2} \cdots a^{\epsilon_{p-2}} b^{\epsilon_{p-1}}$.
\end{enumerate}
\end{lemma}

\begin{remark} \label{epsilon}
{\rm
For $r=0/1$ and $r=1/0$, we have $u_{0/1}=ab$ and $u_{1/0}=1$.
}
\end{remark}

\subsection{$S$- and $T$-sequences of slope $r$}

We define the sequence $S(r)$ and the cyclic sequence $CS(r)$ of slope $r$
both of which arise from the single relator $u_r$ of the upper presentation
$G(K(r))=\langle a, b \, | \, u_r \rangle$,
and review several important properties of these sequences from \cite{lee_sakuma}
so that we can adapt small cancellation theory in Section~\ref{annular_diagrams}.

We first fix some definitions and notation.
Let $X$ be a set.
By a {\it word} in $X$, we mean a finite sequence
$x_1^{\epsilon_1}x_2^{\epsilon_2}\cdots x_n^{\epsilon_n}$
where $x_i\in X$ and $\epsilon_i=\pm1$.
Here we call $x_i^{\epsilon_i}$ the {\it $i$-th letter} of the word.
For two words $u, v$ in $X$, by
$u \equiv v$ we denote the {\it visual equality} of $u$ and
$v$, meaning that if $u=x_1^{\epsilon_1} \cdots x_n^{\epsilon_n}$
and $v=y_1^{\delta_1} \cdots y_m^{\delta_m}$ ($x_i, y_j \in X$; $\epsilon_i, \delta_j=\pm 1$),
then $n=m$ and $x_i=y_i$ and $\epsilon_i=\delta_i$ for each $i=1, \dots, n$.
For example, two words $x_1x_2x_2^{-1}x_3$ and $x_1x_3$ ($x_i \in X$) are {\it not} visually equal,
though $x_1x_2x_2^{-1}x_3$ and $x_1x_3$ are equal as elements of the free group with basis $X$.
The length of a word $v$ is denoted by $|v|$.
A word $v$ in
$X$ is said to be {\it reduced} if $v$ does not contain $xx^{-1}$ or $x^{-1}x$ for any $x \in X$.
A word is said to be {\it cyclically reduced}
if all its cyclic permutations are reduced.
A {\it cyclic word} is defined to be the set of all cyclic
permutations of a word.
By $(v)$ we denote the cyclic word associated with a word $v$.
Also by $(u) \equiv (v)$ we mean the {\it visual equality} of two cyclic words
$(u)$ and $(v)$. In fact, $(u) \equiv (v)$ if and only if $v$ is visually a cyclic shift
of $u$.

\begin{definition}
\label{def:alternating}
{\rm (1) Let $v$ be a reduced word in
$\{a,b\}$. Decompose $v$ into
\[
v \equiv v_1 v_2 \cdots v_t,
\]
where, for each $i=1, \dots, t-1$, all letters in $v_i$ have positive (resp., negative) exponents,
and all letters in $v_{i+1}$ have negative (resp., positive) exponents.
Then the sequence of positive integers
$S(v):=(|v_1|, |v_2|, \dots, |v_t|)$ is called the {\it $S$-sequence of $v$}.

(2) Let $(v)$ be a reduced
cyclic word in $\{a, b\}$. Decompose $(v)$ into
\[
(v) \equiv (v_1 v_2 \cdots v_t),
\]
where all letters in $v_i$ have positive (resp., negative) exponents,
and all letters in $v_{i+1}$ have negative (resp., positive) exponents (taking
subindices modulo $t$). Then the {\it cyclic} sequence of positive integers
$CS(v):=\lp |v_1|, |v_2|, \dots, |v_t| \rp$ is called
the {\it cyclic $S$-sequence of $(v)$}.
Here the double parentheses denote that the sequence is considered modulo
cyclic permutations.

(3) A reduced word $v$ in $\{a,b\}$ is said to be {\it alternating}
if $a^{\pm 1}$ and $b^{\pm 1}$ appear in $v$ alternately,
i.e., neither $a^{\pm2}$ nor $b^{\pm2}$ appears in $v$.
A cyclic word $(v)$ is said to be {\it alternating}
if all cyclic permutations of $v$ are alternating.
In the latter case, we also say that $v$ is {\it cyclically alternating}.
}
\end{definition}

\begin{definition}
\label{def4.1(3)}
{\rm
For a rational number $r$ with $0<r\le 1$,
let $G(K(r))=\langle a, b \, | \, u_r \rangle$
be the upper presentation.
Then the symbol $S(r)$ (resp., $CS(r)$) denotes the
$S$-sequence $S(u_r)$ of $u_r$
(resp., cyclic $S$-sequence $CS(u_r)$ of $(u_r)$), which is called
the {\it S-sequence of slope $r$}
(resp., the {\it cyclic S-sequence of slope $r$}).}
\end{definition}

\begin{lemma}{\rm \cite[Proposition~4.1 and Lemma~5.2]{lee_sakuma}}
\label{meaning_of_S-sequence+initial}
{\rm (1)} An alternating word
in $\{a,b\}$ is completely determined by the
initial letter
and the associated $S$-sequence.

{\rm (2)} Suppose that $r$ is a rational number with $0 < r \le 1$.
Let $w$ be an arbitrary cyclic permutation
of the single relator $u_r$ of the upper presentation of $G(K(r))$.
Then the set
\[
\{ \text{the initial letter of $w' \, | \, (w') \equiv (u_r^{\pm 1})$ and $S(w')=S(w)$} \}
\]
equals $\{a, a^{-1}, b, b^{-1} \}$.
\end{lemma}

The above lemma implies that
the cyclic word $(u_r^{\pm 1})$ is completely determined
by the cyclic $S$-sequence $CS(r)$, to be precise,
an alternating word $w'$ is a cyclic permutation of $u_r^{\pm 1}$
if and only if
$w'$ satisfies $S(w')=S(w)$ for some $w$
a cyclic permutation of $u_r$.

\begin{lemma}{\rm \cite[Lemma~4.8 and Proposition~4.2]{lee_sakuma}}
\label{j-term}
Suppose that $r=q/p$ is a rational number with $0<r \le 1$,
where $p$ and $q$ are relatively prime positive integers.
Then the following hold.
\begin{enumerate}[\indent \rm (1)]
\item The sequence $S(r)$ has length $2q$, and its
$j$-th term $\nu_j(r)$ is given by the following formula
$(1\le j\le 2q)${\rm :}
\[
\nu_j(r)=\lfloor jp/q \rfloor_*- \lfloor (j-1)p/q \rfloor_*,
\]
where $\lfloor x \rfloor_*$ is the greatest integer
(strictly) smaller than $x$.

\item The sequence $S(r)$ represents the cyclic sequence $CS(r)$.
Moreover the cyclic sequence $CS(r)$ is invariant by
the half-rotation; that is,
if $\nu_j(r)$ denotes the $j$-th term of $S(r)$
$(1\le j\le 2q)$, then
$\nu_j(r)=\nu_{q+j}(r)$ for every integer $j$ $(1\le j\le q)$.
\end{enumerate}
\end{lemma}

\begin{remark}
\label{remark:j-term}
\rm
For $r=0$, we have $S(u_0)=(2)$ by Remark~\ref{epsilon}.
\end{remark}

In the remainder of this paper unless specified otherwise,
we suppose that $r$ is a rational number
with $0<r \le 1$, and write $r$ as a continued fraction:
\[
r=[m_1,m_2, \dots,m_k],
\]
where $k \ge 1$, $(m_1, \dots, m_k) \in (\mathbb{Z}_+)^k$ and
$m_k \ge 2$ unless $k=1$. For brevity, we write $m$ for $m_1$.

\begin{lemma}{\rm \cite[Proposition~4.3]{lee_sakuma}}
\label{properties}
The following hold.
\begin{enumerate}[\indent \rm (1)]
\item Suppose $k=1$, i.e., $r=1/m$.
Then $S(r)=(m,m)$.

\item Suppose $k\ge 2$. Then each term of $S(r)$ is either $m$ or $m+1$,
and $S(r)$ begins with $m+1$ and ends with $m$.
Moreover, the following hold.

\begin{enumerate}[\rm (a)]
\item If $m_2=1$, then no two consecutive terms of $S(r)$ can be $(m, m)$,
so there is a sequence of positive integers $(t_1,t_2,\dots,t_s)$ such that
\[
S(r)=(t_1\langle m+1\rangle, m, t_2\langle m+1\rangle, m, \dots,
t_s\langle m+1\rangle, m).
\]
Here, the symbol ``$t_i\langle m+1\rangle$'' represents $t_i$ successive $m+1$'s.

\item If $m_2 \ge 2$, then no two consecutive terms of $S(r)$ can be $(m+1, m+1)$,
so there is a sequence of positive integers $(t_1,t_2,\dots,t_s)$ such that
\[
S(r)=(m+1, t_1\langle m\rangle, m+1, t_2\langle m\rangle,
\dots,m+1, t_s\langle m\rangle).
\]
Here, the symbol ``$t_i\langle m\rangle$'' represents $t_i$ successive $m$'s.
\end{enumerate}
\end{enumerate}
\end{lemma}

\begin{definition}
\label{def_T(r)}
{\rm
If $k\ge 2$, the symbol $T(r)$ denotes the sequence
$(t_1,t_2,\dots,t_s)$ in Lemma~\ref{properties},
which is called the {\it $T$-sequence of slope $r$}.
The symbol $CT(r)$ denotes the cyclic
sequence represented by $T(r)$, which is called the
{\it cyclic $T$-sequence of slope $r$}.
}
\end{definition}

\begin{example}
\label{cyclic_sequence}
{\rm (1) Let $r={10}/{37}=[3,1,2,3]$.
By Lemma~\ref{presentation}, we
see that the $S$-sequence of $\hat{u}_r$ is
\[
S(\hat{u}_r)=(3, 4, 4, 3, 4, 4, 3, 4, 4, 3).
\]
By the formula for $u_r$ in Lemma~\ref{presentation},
this implies
\[
S(r)=S(u_r) =
(\underbrace{4,4,4}_3,3,\underbrace{4, 4}_2, 3, \underbrace{4, 4}_2, 3,
\underbrace{4,4,4}_3,3, \underbrace{4, 4}_2, 3, \underbrace{4, 4}_2, 3).
\]
So $T(r)=(3, 2, 2, 3, 2, 2)$ and $CT(r) = \lp 3, 2, 2, 3, 2, 2 \rp$.

(2) Let $r={8}/{35}=[4,2,1,2]$.
Again by Lemma~\ref{presentation},
we obtain that the $S$-sequence of $\hat{u}_r$ is
\[
S(\hat{u}_r)=(4, 4, 5, 4, 4, 5, 4, 4).
\]
By the formula for $u_r$ in Lemma~\ref{presentation},
this implies
\[
S(r)=S(u_r)=
(5, \underbrace{4}_1, 5,
\underbrace{4, 4}_2, 5, \underbrace{4, 4}_2, 5, \underbrace{4}_1, 5,
\underbrace{4, 4}_2, 5, \underbrace{4, 4}_2).
\]
So $T(r) = (1, 2, 2, 1, 2, 2)$ and $CT(r) = \lp 1, 2, 2, 1, 2, 2 \rp$.
}
\end{example}

\begin{lemma}{\rm \cite[Proposition~4.4]{lee_sakuma}}
\label{induction1}
Let $\tilde{r}$ be the rational number defined as
\[
\tilde{r}=
\begin{cases}
[m_3, \dots, m_k] & \text{if $m_2=1$};\\
[m_2-1, m_3, \dots, m_k] & \text{if $m_2 \ge 2$}.
\end{cases}
\]
Then we have
\[
T(r)=
\begin{cases}
S(\tilde{r}) & \text{if $m_2=1$}; \\
\overleftarrow{S}(\tilde{r}) & \text{if $m_2 \ge 2$},
\end{cases}
\]
where $\overleftarrow{S}(\tilde{r})$ denotes the sequence obtained from
$S(\tilde{r})$ reversing its order.
\end{lemma}

\begin{proposition}{\rm \cite[Proposition~4.5]{lee_sakuma}}
\label{sequence}
The sequence $S(r)$ has a unique
decomposition $(S_1, S_2, S_1, S_2)$ which satisfies the following.
\begin{enumerate} [\indent \rm (1)]
\item Each $S_i$ is symmetric,
i.e., the sequence obtained from $S_i$ by reversing the order is
equal to $S_i$. {\rm (}Here, $S_1$ is empty if $k=1$.{\rm )}
\item Each $S_i$
{\rm (}if it is not empty{\rm )}
occurs only twice in the cyclic sequence $CS(r)$.
\item $S_1$
{\rm (}if it is not empty{\rm )}
begins and ends with $m+1$.
\item $S_2$ begins and ends with $m$.
\end{enumerate}
\end{proposition}

\begin{example}
{\rm (1) Let $r={10}/{37}=[3,1,2,3]$.
Recall from Example~\ref{cyclic_sequence} that
\[
S(r) =(4,4,4, 3, 4, 4, 3, 4, 4, 3, 4,4,4,3,4, 4, 3, 4, 4, 3).
\]
Putting $S_1=(4,4,4)$ and $S_2=(3, 4, 4, 3, 4, 4, 3)$, we have
\[
S(r)= (S_1, S_2, S_1, S_2),
\]
where $S_1$ and $S_2$ satisfy all the assertions in Proposition~\ref{sequence}.

(2) Let $r={8}/{35}=[4,2,1,2]$. Recall also from
Example~\ref{cyclic_sequence} that
\[
S(r) = (5, 4, 5, 4, 4, 5, 4, 4, 5, 4, 5, 4, 4, 5, 4, 4).
\]
Putting $S_1=(5, 4, 5)$ and $S_2=(4, 4, 5, 4, 4)$, we also have
\[
S(r) = (S_1, S_2, S_1, S_2),
\]
where $S_1$ and $S_2$ satisfy all the assertions in Proposition~\ref{sequence}.
}
\end{example}

\begin{corollary}{\rm \cite[Corollary~4.6]{lee_sakuma}}
\label{induction2}
The cyclic $S$-sequence $CS(r)$ is symmetric,
i.e., the cyclic sequence obtained from $CS(r)$
by reversing its cyclic order is equivalent to $CS(r)$
{\rm (}as a cyclic sequence{\rm )}.
In particular, in Lemma~\ref{induction1}, we actually have
\[
CT(r)=CS(\tilde{r}).
\]
\end{corollary}

\begin{remark}
\label{remark:recovering the slope}
{\rm
By using the fact that $u_r$ is obtained from
the line of slope $r$ in $\RR^2-\ZZ^2$
by reading its intersection with the vertical lattice lines,
we see that the slope $s=q/p$ is recovered from $CS(s)=\lp S_1,S_2,S_1,S_2 \rp$
by the rule that $p$ is the sum of the
terms of $S_1$ and $S_2$
whereas $q$ is the sum of the lengths of $S_1$ and $S_2$.
}
\end{remark}

\begin{lemma}{\rm \cite[Proof of Proposition~4.5]{lee_sakuma}}
\label{relation}
Let $\tilde{r}$
be the rational number defined as in Lemma~\ref{induction1}.
Also let
$S(\tilde{r})=(T_1, T_2, T_1, T_2)$
and $S(r) =(S_1, S_2, S_1, S_2)$
be decompositions described as in Proposition~\ref{sequence}.
Then the following hold.
\begin{enumerate} [\indent \rm (1)]
\item If $m_2=1$ and $k=3$, then
$T_1=\emptyset$, $T_2=(m_3)$,
and $S_1 =(m_3\langle m+1 \rangle)$, $S_2 =(m)$.

\item If $m_2=1$ and $k\ge 4$, then
$T_1=(t_1, \dots, t_{s_1})$, $T_2=(t_{s_1+1}, \dots, t_{s_2})$,
and
\[
\begin{aligned}
S_1
&=(t_1 \langle m+1 \rangle, m,
t_2 \langle m+1 \rangle,
\dots, t_{s_1-1}\langle m+1 \rangle, m,
t_{s_1}\langle m+1 \rangle), \\
S_2 &=(m, t_{s_1+1}\langle m+1\rangle,m, \dots,
m, t_{s_2}\langle m+1\rangle, m).
\end{aligned}
\]

\item If $k=2$, then $T_1=\emptyset$, $T_2=(m_2-1)$,
and $S_1 =(m+1)$, $S_2 =((m_2-1)\langle m \rangle)$.

\item If $m_2 \ge 2$ and $k\ge 3$, then
$T_1=(t_1, \dots, t_{s_1})$,
$T_2=(t_{s_1+1}, \dots, t_{s_2})$,
and
\[
\begin{aligned}
S_1 &=
(m+1, t_{s_1+1}\langle m\rangle, m+1,
\dots, m+1, t_{s_2}\langle m\rangle, m+1),\\
S_2 &=
(t_1\langle m\rangle, m+1,t_2\langle m\rangle, \dots,
t_{s_1-1}\langle m\rangle, m+1, t_{s_1}\langle m\rangle).
\end{aligned}
\]
\end{enumerate}
\end{lemma}

The following example will be useful in the second~\cite{lee_sakuma_3}
of this series of papers.

\begin{example}
\label{example:S-sequence}
{\rm
(1) If $r=[2, 1, n]$ with $n \ge 2$, then
$\tilde{r}=[n]$ by Lemma~\ref{induction1}.
So by Lemma~\ref{properties}(1),
$S(\tilde{r})=(\emptyset,n,\emptyset,n)$.
Thus by Lemma~\ref{relation}(1),
$S(r)=(n \langle 3 \rangle, 2, n \langle 3 \rangle, 2)$,
where $S_1=(n \langle 3 \rangle)$ and $S_2=(2)$.

(2) If $r=[2, n]$ with $n \ge 2$, then
$\tilde{r}=[n-1]$ by Lemma~\ref{induction1}.
So by Lemma~\ref{properties}(1),
$S(\tilde{r})=(\emptyset,n-1,\emptyset,n-1)$.
Thus by Lemma~\ref{relation}(3),
$S(r)=(3, (n-1) \langle 2 \rangle, 3, (n-1) \langle 2 \rangle)$,
where $S_1=(3)$ and $S_2=((n-1) \langle 2 \rangle)$.
}
\end{example}

By Lemmas~\ref{properties} and \ref{relation},
we can easily observe the following lemma.

\begin{lemma}
\label{cor-to-relation}
Let $S(r)=(S_1,S_2,S_1,S_2)$ be as in Proposition \ref{sequence}.
Then the following hold.
\begin{enumerate} [\indent \rm (1)]
\item If $m_2=1$, then $(m+1,m+1)$ appears in $S_1$.

\item If $m_2\ge 2$ and if $r \ne [m,2]=2/(2m+1)$,
then $(m,m)$ appears in $S_2$.
\end{enumerate}
\end{lemma}

The following is a refinement of \cite[Lemma~7.3 and Remark~5]{lee_sakuma}.

\begin{proposition}
\label{connection}
Suppose that $r$ is a rational number with $0<r<1$.
Let $S(r)= (S_1, S_2, S_1, S_2)$ be as in Proposition~\ref{sequence}.
For a rational number $s$ with $0 < s \le 1$, the following hold.
\begin{enumerate} [\indent \rm (1)]
\item If $CS(s)$ contains both $S_1$ and $S_2$ as subsequences,
then $s \notin I_1(r) \cup I_2(r)$.

\item If $r=1/p$ $(p \ge 2)$ and $s \in I_1(r) \cup I_2(r)$,
then $CS(s)$ consists of integers less than $p$.
\end{enumerate}
\end{proposition}

In the above proposition (and throughout this series of papers),
we mean by a {\it subsequence}
a subsequence without leap.
Namely a sequence $(a_1,a_2,\dots, a_l)$
is called a {\it subsequence} of a cyclic sequence,
if there is a sequence $(b_1,b_2,\dots, b_n)$
representing the cyclic sequence
such that $l \le n$ and
$a_i=b_i$ for $1\le i\le l$.

\begin{proof}
(1) Suppose that $CS(s)$ contains both $S_1$ and $S_2$
as subsequences.
Recall that $r=[m_1,m_2, \dots, m_k]$.
Write $s=[n_1,n_2, \dots, n_t]$,
where $t \ge 1$, $(n_1, \dots, n_t) \in (\mathbb{Z}_+)^t$ and
$n_t \ge 2$ unless $t=1$.
We show that the following three conditions hold
by induction on $k\ge 1$,
refining the proof of \cite[Lemma~7.3]{lee_sakuma}.
(As noted in \cite[Remark~5]{lee_sakuma},
this is equivalent to the desired conclusion
that $s \notin I_1(r) \cup I_2(r)$.)
\begin{enumerate}[\indent \rm (i)]
\item $t \ge k$.

\item $n_i=m_i$ for each $i=1, \dots, k-1$.

\item Either $n_k \ge m_k$ or both $n_k=m_k-1$ and $t>k$.
\end{enumerate}

First let $k=1$. Then $r=[m]$ and $S(r)=(m, m)=(S_2, S_2)$,
where $S_1$ is empty. By hypothesis, $CS(s)$ contains a term $m$.
So if $t=1$, then $n_1=m$,
while if $t \ge 2$, then $n_1$ is either $m$ or $m-1$.
Thus the three conditions hold, proving the base step.

Now let $k \ge 2$.
Then $CS(r)=\lp S_1, S_2, S_1, S_2 \rp$ consists of $m$ and $m+1$
by Lemma~\ref{properties}.
This yields that $CS(s)$ consists of $m$ and $m+1$.
This happens only when $t \ge 2$ and $n_1=m_1$.
For the rational numbers $r$ and $s$, define the rational numbers
$\tilde{r}$ and $\tilde{s}$ as in Lemma~\ref{induction1}.

We consider three cases separately.

\medskip
\noindent {\bf Case 1.} {\it $m_2=1$.}
\medskip

In this case, $k \ge 3$ and,
by Lemma~\ref{cor-to-relation},
$(m+1, m+1)$ appears in $S_1$ as a subsequence,
so in $CS(s)$ as a subsequence.
Thus by Lemma~\ref{properties},
$n_2=1$ and so $t \ge 3$.
So, we have
\[
\tilde{r}=[m_3, \dots, m_k] \quad \text{\rm and} \quad
\tilde{s}=[n_3, \dots, n_t].
\]
Let $S(\tilde{r})=(T_1, T_2, T_1, T_2)$ be
the decomposition of $S(\tilde{r})$
given by Proposition~\ref{sequence}.

\medskip
\noindent {\bf Case 1.a.} {\it $k=3$.}
\medskip

By Lemma~\ref{relation}(1),
$S_1 =(m_3\langle m+1 \rangle)$ and $S_2 =(m)$.
Since $S_1$ is contained in $CS(s)$ by assumption,
$CS(\tilde{s})=CT(s)$ contains a term
$m_3+d$ for some $d\in\ZZ_+\cup\{0\}$.
If $t=3$ then $CS(\tilde{s})=\lp n_3,n_3 \rp$ and hence
we have $n_3\ge m_3$.
If $t\ge 4$ then $CS(\tilde{s})$ consists of
$n_3$ and $n_3+1$ and hence $n_3\ge m_3-1$.
Thus in either case, the three conditions hold, as desired.

\medskip
\noindent {\bf Case 1.b.} {\it $k\ge 4$.}
\medskip

Since $S_2$ is contained in $CS(s)$ by assumption
and since $S_2$ begins and ends with $m$,
we see by Lemma~\ref{relation}(2) that
$CS(\tilde{s})=CT(s)$ contains $T_2$.
Similarly, by using the assumption that
$S_1$ is contained in $CS(s)$,
we see that $CS(\tilde{s})=CT(s)$ contains
a subsequence of the form
\[
(t_1+d', t_2, \dots, t_{s_1-1}, t_{s_1}+d''),
\]
where
$(t_1, t_2, \dots, t_{s_1-1},t_{s_1})=T_1$ and
$d', d''\in\ZZ_+\cup\{0\}$.
Since $t_1=t_{s_1}=m_3+1$ by
Proposition~\ref{sequence},
this actually implies that
$CS(\tilde{s})$ contains $T_1$ as a subsequence.
Thus $CS(\tilde{s})$ contains both $T_1$ and $T_2$
as subsequences.
Since $S(\tilde{r})=(T_1,T_2,T_1,T_2)$ is the decomposition
described as in Proposition~\ref{sequence},
the inductive hypothesis implies that
the following three conditions hold.
\begin{enumerate}[\indent \rm (i)]
\item $t \ge k$.

\item $n_i=m_i$ for each $i=3, \dots, k-1$.

\item Either $n_k \ge m_k$ or both $n_k=m_k-1$ and $t>k$.
\end{enumerate}
Since $n_1=m_1$ and $n_2=m_2$, this implies that
the original three conditions hold, as desired.

\medskip
\noindent {\bf Case 2.} {\it Both $m_2=2$ and $k=2$.}
\medskip

In this case, the three conditions always hold,
because if $n_2=1$ then we must have $t \ge 3$, otherwise $n_2 \ge 2=m_2$.
(Recall that we already proved that $n_1=m_1$.)

\medskip
\noindent {\bf Case 3.} {\it Either $m_2 \ge 3$ or both $m_2=2$ and $k \ge 3$.}
\medskip

In this case, by Lemma~\ref{cor-to-relation},
$(m,m)$ appears in $S_2$ as a subsequence,
so in $CS(s)$ as a subsequence.
Thus $n_2\ge 2$ by Lemma~\ref{properties}, and so we have
\[
\tilde{r}=[m_2-1,m_3, \dots, m_k] \quad \text{\rm and} \quad
\tilde{s}=[n_2-1,n_3, \dots, n_t].
\]
Let $S(\tilde{r})= (T_1, T_2, T_1, T_2)$ be
the decomposition of $S(\tilde{r})$
given by Proposition~\ref{sequence}.
Since $S_1$ is contained in $CS(s)$ by assumption
and since $S_1$ begins and ends with $m+1$,
we see by Lemma~\ref{relation}(4) that
$CS(\tilde{s})=CT(s)$ contains $T_2$.
Similarly, by using the assumption that
$S_2$ is contained in $CS(s)$,
we see that $CS(\tilde{s})=CT(s)$ contains
a subsequence of the form
\[
(t_{1}+d', t_2, \dots, t_{s_1-1},t_{s_1}+d''),
\]
where
$(t_{1}, t_2, \dots, t_{s_1-1},t_{s_1})=T_1$ and
$d', d''\in\ZZ_+\cup\{0\}$.
Since $t_{1}=t_{s_1}=(m_2-1)+1=m_2$ by
Proposition~\ref{sequence},
this actually implies that
$CS(\tilde{s})$ contains $T_1$ as a subsequence.
Thus $CS(\tilde{s})$ contains both $T_1$ and $T_2$
as subsequences.
Since $S(\tilde{r})=(T_1,T_2,T_1,T_2)$ is the decomposition
described as in Proposition~\ref{sequence},
the inductive hypothesis implies that
the following three conditions hold.
\begin{enumerate}[\indent \rm (i)]
\item $t \ge k$.

\item $n_i=m_i$ for each $i=2, \dots, k-1$.

\item Either $n_k \ge m_k$ or both $n_k=m_k-1$ and $t>k$.
\end{enumerate}
Since $n_1=m_1$,
this implies that the original three conditions hold, as desired.

(2) Assume $r=1/p$.
Since $0 \neq s \in I_1(1/p) \cup I_2(1/p)=\{0\} \cup [{1 \over {p-1}}, 1]$,
$s$ has a continued
fraction expansion $s=[n_1, \dots, n_t]$ such that
either both $n_1=p-1$ and $t=1$ or both $0< n_1 \le p-2$ and $t \ge 1$.
If $n_1=p-1$ and $t=1$, then $s=1/(p-1)$ and $CS(s)= \lp p-1,p-1 \rp$;
so the assertion holds.
If $0< n_1 \le p-2$ and $t \ge 1$, then
each term of $CS(s)$ is equal to $n_1 \le p-2$
or $n_1+1\le p-1$ by Lemma~\ref{properties};
so the assertion holds.
\end{proof}

\begin{remark}
\label{rem:extreme_case}
\rm
If $r\ne 1/2$, then the conclusions of
Proposition~\ref{connection}
hold even for $s=0$. In fact, if $r\ne 1/2$, then either
(i) $r=1/p$ with $p \ge 3$ and so $S(r)=(p, p)=(S_2, S_2)$,
where $S_1$ is empty, or
(ii) both $S_1$ and $S_2$ are non-empty.
Since $CS(u_0)=\lp 2 \rp$ by Remark~\ref{remark:j-term},
this implies the conclusions of
Proposition~\ref{connection} for $s=0$.
However, if $r=1/2$, then
none of the assertions of Proposition~\ref{connection}
holds for $s=0$.
\end{remark}

\subsection{Small cancellation conditions}

We now recall the small cancellation conditions for
$2$-bridge link groups established in \cite{lee_sakuma}.

Let $F(X)$ be the free group with basis $X$. A subset $R$ of $F(X)$
is said to be {\it symmetrized},
if all elements of $R$ are cyclically reduced and, for each $w \in R$,
all cyclic permutations of $w$ and $w^{-1}$ also belong to $R$.

\begin{definition}
{\rm Suppose that $R$ is a symmetrized subset of $F(X)$.
A nonempty word $b$ is called a {\it piece} if there exist distinct $w_1, w_2 \in R$
such that $w_1 \equiv bc_1$ and $w_2 \equiv bc_2$.
The small cancellation conditions $C(i)$ and $T(j)$ on $R$,
where $i$ and $j$ are integers such that $i \ge 2$ and $j \ge 3$,
are defined as follows (see \cite{lyndon_schupp}).
\begin{enumerate}[\indent \rm (1)]
\item Condition $C(i)$: If $w \in R$
is a product of $n$ pieces, then $n \ge i$.

\item Condition $T(j)$: For
$w_1, \dots, w_n \in R$
with no successive elements
$w_t, w_{t+1}$
an inverse pair $(t$ mod $n)$, if $n < j$, then at least one of the products
$w_1 w_2,\dots,$ $w_{n-1} w_n$, $w_n w_1$
is freely reduced without cancellation.
\end{enumerate}
}
\end{definition}

The following proposition enables us to
apply the small cancellation theory to our problem.

\begin{proposition}{\rm \cite[Theorem~5.1]{lee_sakuma}}
\label{small_cancellation_condition}
Suppose that $r$ is a rational number with $0 < r< 1$.
Let $R$ be the symmetrized subset of $F(a, b)$ generated
by the single relator $u_{r}$ of the upper presentation of $G(K(r))$.
Then $R$ satisfies $C(4)$ and $T(4)$.
\end{proposition}

To conclude this section,
we recall a key fact concerning the cyclic word $(u_r)$,
which is used in the proof of the main theorem in
Section~\ref{sec:only_if_part}.
(In fact, it is also used in the proof of
Proposition~\ref{small_cancellation_condition} above
and that of Lemma~\ref{degree2_vertices} implicitly.)

\begin{definition}
\label{def:n-piece}
{\rm
For a positive integer $n$,
a non-empty subword $w$ of the cyclic word $(u_r)$
is called a {\it maximal $n$-piece}
if $w$ is a product of $n$ pieces and
if any subword $w'$ of $(u_r)$
which properly contains $w$ as an {\it initial} subword
is not a product of $n$ pieces.}
\end{definition}

\begin{lemma}{\rm \cite[Corollary~5.4]{lee_sakuma}}
\label{max_piece}
Suppose that $r$ is a rational number such that $0 < r <1$.
Let $u_r$ be the single relator of the upper presentation of $G(K(r))$,
and let $S(r)=(S_1, S_2, S_1, S_2)$ be as in Proposition~\ref{sequence}.
Decompose
\[
u_r \equiv v_1 v_2 v_3 v_4,
\]
where $S(v_1)=S(v_3)=S_1$ and $S(v_2)=S(v_4)=S_2$.
Let $v_{ib}^*$
be the maximal proper initial subword of $v_i$,
i.e., the initial subword of $v_i$
such that $|v_{ib}^*|=|v_i|-1$ $(i=1,2,3,4)$.
Then the following hold,
where $v_{ib}$ and $v_{ie}$ are nonempty initial and terminal subwords of $v_i$
with $|v_{ib}|, |v_{ie}| \le |v_i|-1$, respectively.

\begin{enumerate}[\indent \rm (1)]
\item If $r=1/p$ for some integer $p \ge 2$, then
$v_1$ and $v_3$ are the empty words
and the following hold.

\begin{enumerate}[\rm (a)]

\item The following is the list of all
maximal $1$-pieces of $(u_r)$,
arranged in the order of
the position of the initial letter:
\[
v_{2b}^*, v_{2e}, v_{4b}^*, v_{4e}.
\]

\item The following is the list of all
maximal $2$-pieces of $(u_r)$,
arranged in the order of
the position of the initial letter:
\[
v_2, v_{2e} v_{4b}^*, v_4, v_{4e} v_{2b}^*.
\]
\end{enumerate}

\item If $r \neq 1/p$ for any integer $p \ge 2$, then the following hold.

\begin{enumerate}[\rm (a)]

\item The following is the list of all
maximal $1$-pieces of $(u_r)$,
arranged in the order of
the position of the initial letter:
\[
v_{1b}^*, v_{1e} v_2, v_2 v_{3b}^*, v_{2e} v_{3b}^*,
v_{3b}^*, v_{3e} v_4, v_4 v_{1b}^*, v_{4e} v_{1b}^*.
\]

\item The following is the list of all
maximal $2$-pieces of $(u_r)$,
arranged in the order of
the position of the initial letter:
\[
v_1 v_2, v_{1e} v_2 v_{3b}^*,
v_2 v_{3} v_4,
v_{2e} v_3 v_4,
v_3 v_4, v_{3e} v_4 v_{1b}^*,
v_4 v_{1} v_2,
v_{4e} v_1 v_2.
\]
\end{enumerate}
\end{enumerate}
\end{lemma}

The following corollary will be used in the sequels
\cite{lee_sakuma_3} and \cite{lee_sakuma_4} of the present paper.

\begin{corollary}
\label{cor:max_piece_2}
If $r$ is a rational number such that $0<r<1$ and
$r \neq 1/p$ for any integer $p \ge 2$, then the following hold.

\begin{enumerate}[\indent \rm (1)]
\item A subword $w$ of the cyclic word
$(u_r^{\pm 1})$ is a piece if and only if
$S(w)$ does not contain $S_1$ as a subsequence
and does not contain $S_2$ in its interior,
i.e., $S(w)$ does not contain a subsequence
$(\ell_1, S_2, \ell_2)$ for
any $\ell_1,\ell_2\in\ZZ_+$.

\item For a subword $w$ of the cyclic word $(u_r^{\pm 1})$,
if $S(w)$ either contains
$(S_1,S_2)$ as a proper initial subsequence
or contains
$(S_2,S_1)$ as a proper terminal subsequence,
then $w$ is not a product of two pieces.
\end{enumerate}
\end{corollary}

\begin{proof}
(1) We prove the assertion for a subword of
the cyclic word $(u_r)$.
(The assertion for a subword of
the cyclic word $(u_r^{-1})$ follows from this
and the facts that
$S_1$ and $S_2$ are symmetric and that
$w$ is a piece if and only if $w^{-1}$ is a piece.)
To show the only if part, let $w$ be a
piece which is a subword of $(u_r)$.
Suppose on the contrary that
$S(w)$ contains $S_1$ or $(\ell_1, S_2, \ell_2)$
($\ell_1,\ell_2\in\ZZ_+$) as a subsequence.
Let $w'$ be a subword of $w$ corresponding to
the subsequence.
Then, since each of $S_1$ and $S_2$ appears
only twice in the cyclic sequence $CS(r)$
by Proposition~\ref{sequence}(2),
we have the following.
\begin{enumerate} [\indent \rm (i)]
\item If $S(w')=S_1$, then $w'=v_1$ or $v_3$.

\item If $S(w')=(\ell_1, S_2, \ell_2)$,
then $w'=v_{1e}v_2v_{3b}$ or $v_{3e}v_4v_{1b}$.
\end{enumerate}
In either case, $w'$ cannot be a subword of any of the
maximal $1$-pieces of $(u_r)$ listed in Lemma~\ref{max_piece}(2a).
Hence, $w$ is not a piece, a contradiction.
To see the if part, let $w$ be a subword of $(u_r)$
whose $S$-sequence
does not contain $S_1$ nor $(\ell_1, S_2, \ell_2)$
($\ell_1,\ell_2\in\ZZ_+$) as a subsequence.
Then we see by using Proposition~\ref{sequence}(2) that
$w$ does not contain $v_1$, $v_3$, $v_{1e}v_2v_{3b}$, nor
$v_{3e}v_4v_{1e}$ as a subword.
Since $w$ is a subword of $(u_r)$,
this implies that
$w$ is a subword of one of the
maximal $1$-pieces of $(u_r)$ listed in Lemma~\ref{max_piece}(2a).
Hence $w$ is a piece.

(2) As in (1), we prove the assertion for a subword of
the cyclic word $(u_r)$.
Suppose that $w$ is a subword of $(u_r)$ such that
$S(w)$ contains either $(S_1,S_2,\ell)$ or $(\ell,S_2,S_1)$
with $\ell \in \ZZ_+$.
Let $w'$ be a subword of $w$ corresponding to
the subsequence.
Then, since $S_1$ and $S_2$ appears
only twice in the cyclic sequence $CS(r)$
by Proposition~\ref{sequence}(2),
we have the following.
\begin{enumerate} [\indent \rm (i)]
\item If $S(w')=(S_1,S_2,\ell)$, then $w'=v_1v_2v_{3b}$ or $v_3v_4v_{1b}$.

\item If $S(w')=(\ell,S_2,S_1)$,
then $w'=v_{1e}v_2v_3$ or $v_{3e}v_4v_1$.
\end{enumerate}
In either case, $w'$ cannot be a subword of any of the
maximal $2$-pieces of $(u_r)$ listed in Lemma~\ref{max_piece}(2b).
Hence, $w$ is not a product of two pieces, as desired.
\end{proof}

\section{Annular diagrams over 2-bridge link groups}
\label{annular_diagrams}

In this section, we establish a very strong
structure theorem (Theorems~\ref{structure} and \ref{cor:structure}) for the annular diagram
which arises in the study of the conjugacy problem
by using the small cancellation theory.
The structure theorem forms
the cornerstone of this whole series of papers.

Let us begin with necessary definitions and notation following \cite{lyndon_schupp}.
A {\it map} $M$ is a finite $2$-dimensional cell complex
embedded in $\RR^2$.
To be precise, $M$ is a finite collection of vertices ($0$-cells), edges ($1$-cells),
and faces ($2$-cells) in $\RR^2$ satisfying the following conditions.
\begin{enumerate}[\indent \rm (i)]
\item A vertex is a point in $\RR^2$.

\item An edge $e$ is homeomorphic to an open interval
such that $\bar e=e\cup\{ a\}\cup \{b\}$,
where $a$ and $b$ are vertices of $M$ which are possibly identical.

\item For each face $D$ of $M$,
there is a continuous map $f$ from the
$2$-ball $B^2$ to $\RR^2$ such that
\begin{enumerate}
\item the restriction of $f$ to the interior of $B^2$
is a homeomorphism onto $D$, and

\item the image of $\partial B^2$ is equal to
$\cup_{i=1}^k \bar e_i$ for some set $\{e_1,\dots, e_k\}$ of edges of $M$.
\end{enumerate}
\end{enumerate}
The underlying space of $M$, i.e., the union of the cells in $M$,
is also denoted by the same symbol $M$.
The boundary (frontier), $\partial M$, of $M$ in $\RR^2$
is regarded as a $1$-dimensional subcomplex of $M$.
An edge may be traversed in either of two directions.
If $v$ is a vertex of a map $M$, $d_M(v)$, the {\it degree of $v$}, will
denote the number of oriented edges in $M$ having $v$ as initial vertex.
A vertex $v$ of $M$ is called an {\it interior vertex}
if $v\not\in \partial M$, and an edge $e$ of $M$ is called
an {\it interior edge} if $e\not\subset \partial M$.

A {\it path} in $M$ is a sequence of oriented edges $e_1, \dots, e_n$ such that
the initial vertex of $e_{i+1}$ is the terminal vertex of $e_i$ for
every $1 \le i \le n-1$. A {\it cycle} is a closed path, namely
a path $e_1, \dots, e_n$
such that the initial vertex of $e_1$ is the terminal vertex of $e_n$.
If $D$ is a face of $M$, any cycle of minimal length which includes
all the edges of the boundary, $\partial D$, of $D$
going around once along the boundary of $D$
is called a {\it boundary cycle} of $D$.
To be precise it is defined as follows.
Let $f: B^2 \rightarrow D$ be a continuous map satisfying
the condition (iii) above.
We may assume that $\partial B^2$ has a cellular structure
such that the restriction of $f$ to each cell is a homeomorphism.
Choose an arbitrary orientation of $\partial B^2$, and let
$\hat e_1, \dots, \hat e_n$ be the oriented edges of $\partial B^2$,
which are oriented in accordance with the orientation of $\partial B^2$
and which lie on $\partial B^2$ in this cyclic order with respect to the orientation of $\partial B^2$.
Let $e_i$ be the orientated edge $f(\hat e_i)$ of $M$.
Then the cycle $e_1, \dots, e_n$, is a boundary cycle of $D$.

\begin{definition}
{\rm A non-empty map $M$ is called a {\it $[p, q]$-map} if the following conditions hold.
\begin{enumerate}[\indent \rm (i)]
\item $d_M(v) \ge p$ for every
interior vertex $v$ of $M$.

\item $d_M(D) \ge q$ for every face $D$ of $M$.
\end{enumerate}
}
\end{definition}

\begin{definition}
{\rm Let $R$ be a symmetrized subset of the free group $F(X)$ with basis $X$.
An {\it $R$-diagram} is a map $M$ and a function $\phi$ assigning to
each oriented edge $e$ of $M$, as a {\it label},
a reduced word $\phi(e)$ in $X$ such that the following hold.
\begin{enumerate}[\indent \rm (i)]
\item If $e$ is an oriented edge of $M$ and $e^{-1}$ is the oppositely oriented edge,
then $\phi(e^{-1})=\phi(e)^{-1}$.

\item For any boundary cycle $\delta$ of any face of $M$,
$\phi(\delta)$ is a cyclically reduced word representing
an element of $R$.
(If $\alpha=e_1, \dots, e_n$ is a path in $M$, we define $\phi(\alpha) \equiv \phi(e_1) \cdots \phi(e_n)$.)
\end{enumerate}
}
\end{definition}

Let $D_1$ and $D_2$ be faces (not necessarily distinct) of $M$
with an edge $e \subseteq \partial D_1 \cap \partial D_2$.
Let $e \delta_1$ and $\delta_2e^{-1}$ be boundary cycles of $D_1$ and $D_2$, respectively.
Let $\phi(\delta_1)=f_1$ and $\phi(\delta_2)=f_2$. An $R$-diagram $M$
is said to be {\it reduced}
if one never has $f_2=f_1^{-1}$.
It should be noted that if $M$ is reduced
then $\phi(e)$ is a piece for every interior edge $e$ of $M$.

\begin{convention}{\rm \cite[Convention~1]{lee_sakuma}}
\label{convention}
{\rm Let $R$ be the symmetrized subset of $F(a, b)$ generated by the single relator $u_r$
of the upper presentation of $G(K(r))$. For any $R$-diagram $M$, we assume
that $M$ satisfies the following.
\begin{enumerate}[\indent \rm (1)]
\item $d_M(v) \ge 3$ for every interior vertex $v$ of $M$.

\item For every edge $e$ of $\partial M$, the label $\phi(e)$ is a piece.

\item For a path $e_1, \dots, e_n$ in $\partial M$ of length $n\ge 2$
such that the vertex $\bar{e}_i\cap \bar{e}_{i+1}$
has degree $2$ for $i=1,2,\dots, n-1$,
$\phi(e_1) \phi(e_2) \cdots\phi(e_n)$ cannot be expressed as a product of fewer than $n$ pieces.
\end{enumerate}
Indeed, we may assume (1), because if there are two interior edges $e_1$ and $e_2$ meeting in
an interior vertex
of degree two, then we can delete the vertex $v$ and unite $e_1$ and $e_2$ into a single edge $e$
with label $\phi(e)=\phi(e_1)\phi(e_2)$.
To see (2), recall that the assumption that $M$ is reduced
implies that $\phi(e)$ is a piece for every interior edge $e$ of $M$.
On the other hand, since the cyclic word $(u_r)$ can be written as a product of pieces,
we may also assume that $\phi(e)$ is a piece for every edge $e$ in $\partial M$.
Finally, we may assume (3),
because if
$\phi(e_1) \cdots\phi(e_n)$
is expressed as a product of less than $n$ pieces,
then we can change the cellular structure
of the interval
$e_1\cup \cdots \cup e_n$
so that the new cellular structure has fewer vertices
compared with the original one.
}
\end{convention}

The following corollary is immediate from
Proposition~\ref{small_cancellation_condition}
and Convention~\ref{convention}.

\begin{corollary}{\rm \cite[Corollary~6.2]{lee_sakuma}}
\label{small_cancellation_condition_2}
Suppose that $r$ is a rational number with $0<r<1$.
Let $R$ be the symmetrized subset of $F(a, b)$ generated
by the single relator $u_{r}$ of the upper presentation of $G(K(r))$.
Then every reduced $R$-diagram is a $[4, 4]$-map.
\end{corollary}

We turn to interpreting conjugacy in terms of diagrams.

\begin{definition}
\label{def:annular_map}
{\rm
An {\it annular map} $M$ is a connected map such that $\RR^2-M$ has exactly two connected components.
It is said to be {\it nontrivial}
if it contains $2$-cells.
For a symmetrized subset $R$ of $F(a, b)$,
an {\it annular $R$-diagram} is an $R$-diagram
whose underlying map is an annular map.
}
\end{definition}

Let $M$ be an annular $R$-diagram, and let $K$ and $H$ be, respectively,
the unbounded and bounded components of $\RR^2-M$. We call
$\partial K (\subset \partial M)$
the {\it outer boundary} of $M$, while
$\partial H (\subset \partial M)$
is called the {\it inner boundary} of $M$.
Clearly, the {\it boundary} of $M$, $\partial M$, is the union of the outer boundary and the inner boundary.
A cycle of minimal length which contains all the edges in the outer (inner, resp.)
boundary of $M$ going around once along the boundary of $K$ ($H$, resp.)
is an {\it outer {\rm (}inner, {\rm resp.)} boundary cycle} of $M$.
An {\it outer {\rm (}inner, {\rm resp.)} boundary label of $M$} is defined to be a word $\phi(\alpha)$ in $X$
for $\alpha$ an outer (inner, resp.) boundary cycle of $M$.

\begin{convention}
\label{convention2}
{\rm
Since $M$ is embedded in $\RR^2$,
each $2$-cell of $M$ inherits an orientation of $\RR^2$.
Throughout this series of papers,
we assume, unlike the usual orientation convention, that
$\RR^2$ is oriented so that the boundary cycles of the $2$-cells of $M$
are clockwise.
Thus the outer boundary cycles are clockwise
and inner boundary cycles are counterclockwise,
unlike the convention in \cite[p.253]{lyndon_schupp}.
}
\end{convention}

The following lemma is a well-known classical result in combinatorial group theory.

\begin{lemma}{\rm \cite[Lemmas~V.5.1 and V.5.2]{lyndon_schupp}}
\label{lyndon_schupp}
Suppose $G=\langle X \,|\, R \, \rangle$ with $R$ being symmetrized.
Let $u, v$ be two cyclically reduced words in $X$
which are not trivial in $G$ and which are not conjugate in $F(X)$.
Then $u$ and $v$ represent conjugate elements in $G$ if and only if
there exists a reduced nontrivial annular $R$-diagram $M$ such that
$u$ is an outer boundary label and $v^{-1}$ is an inner boundary label of $M$.
\end{lemma}

The following lemma will play an essential role
in the proof of Theorem~\ref{structure} below.

\begin{lemma}{\rm \cite[Theorem~V.3.1]{lyndon_schupp}}
\label{inequality}
Let $M$ be an arbitrary connected map. Then
\[
4-4h \le \sum_{v \in \partial M} (3-d_M(v))+ \sum_{v \in M -\partial M} (4-d_M(v))+ \sum_{D \in M} (4-d_M(D)),
\]
where $h$ is the number of holes of $M$,
i.e., the number of bounded components of $\RR^2-M$.
In particular,
if $M$ is a $[4,4]$-map, then
\[
4-4h \le \sum_{v \in \partial M} (3-d_M(v)).
\]
\end{lemma}

In the above lemma and throughout this paper,
the symbol $v \in X$ ($D \in X$, respectively)
under the symbol $\sum$, where $X$ is a map $M$ or a subspace of a map $M$,
means that the sum is over the vertices $v$ (the faces $D$, respectively)
of the map $M$ contained in the subspace $X$.

\begin{proof}
We repeat the proof of a part of \cite[Theorem~V.3.1]{lyndon_schupp}.
Put
\[
\begin{aligned}
V= &\ \text{\rm the number of vertices of $M$;} \\
E= &\ \text{\rm the number of (unoriented) edges of $M$;} \\
F= &\ \text{\rm the number of faces of $M$;} \\
V^{\bullet}= &\ \text{\rm the number of vertices in $\partial M$;} \\
E^{\bullet}= &\ \text{\rm the number of (unoriented) edges in $\partial M$,
{\it counted with multiplicity}.}
\end{aligned}
\]
To be more precise, $E^{\bullet}$ means the number of (unoriented) edges on $\partial M$
with an edge counted twice if it appears twice in the cycles necessary to
describe the boundary of $M$.
For example, the edge $e_0$ in each of Figure~\ref{fig.counting}(a)
and Figure~\ref{fig.counting}(b) has to be counted twice in computing $E^{\bullet}$,
because $e_0$ in Figure~\ref{fig.counting}(a) occurs twice in a boundary cycle of $M$
and $e_0$ in Figure~\ref{fig.counting}(b) occurs in both inner and outer boundary
cycles of $M$. However the vertex $v_0$ in each of Figure~\ref{fig.counting}(a)
and Figure~\ref{fig.counting}(b) has to be counted only once in computing $V^{\bullet}$.

\begin{figure}[h]
\includegraphics{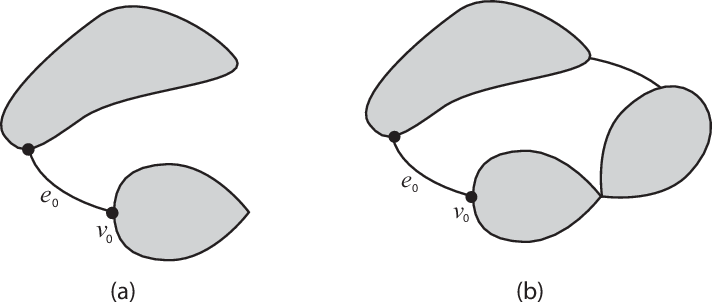}
\caption{
The edge $e_0$ is counted twice in computing $E^{\bullet}$,
while the vertex $v_0$ is counted once in computing $V^{\bullet}$.}
\label{fig.counting}
\end{figure}

Since $M$ has $h$ holes, it follows from Euler's Formula that
\[
1-h=V-E+F.
\tag{\dag}
\label{euler_formula}
\]
Putting $d(v)=d_M(v)$ for vertices $v$ of $M$ and $d(D)=d_M(D)$ for faces $D$ of $M$,
it is easy to observe that
\[
\begin{aligned}
2E&=\sum_{v \in M} d(v); \\
2E&=\sum_{D \in M} d(D) + E^{\bullet}.
\end{aligned}
\]
These equations together with (\ref{euler_formula}) yield
\[
\begin{aligned}
4-4h&=4V+4F-\sum_{v \in M} d(v)- \sum_{D \in M} d(D) - E^{\bullet} \\
&=\sum_{v \in M} (4-d(v)) + \sum_{D \in M} (4-d(D)) - E^{\bullet} \\
&=\sum_{v \in \partial M} (4-d(v)) + \sum_{v \in M -\partial M} (4-d(v))+ \sum_{D \in M} (4-d(D)) - E^{\bullet} \\
&=\sum_{v \in \partial M} (3-d(v)) + \sum_{v \in M -\partial M} (4-d(v))+ \sum_{D \in M} (4-d(D)) + V^{\bullet} - E^{\bullet}.
\end{aligned}
\]
Note that
\[
V^{\bullet} - E^{\bullet}
\le
V^{\bullet} - E^{\bullet}_0
=\chi(\partial M)
=\beta_0(\partial M)-\beta_1(\partial M)
\le 0,
\]
where $E^{\bullet}_0$ is
the number of (unoriented) edges in $\partial M$
(counted {\it without} multiplicity),
$\chi$ denotes the Euler characteristic, and
$\beta_i$ denotes the $i$-th Betti number.
The last inequality follows from the fact that
each component $C$ of $\partial M$ has a positive
first Betti number,
which in turn follows from the fact that
$\RR^2-C$ is not connected.
Hence, the required result follows.
\end{proof}

For a rational number $r$ with $0 < r< 1$,
let $R$ be the symmetrized subset of $F(a, b)$ generated
by the single relator $u_{r}$ of the upper presentation of $G(K(r))$.
Our goal of the present section is to describe the structure of
reduced annular $R$-diagrams with $u_s$ and $u_{s'}^{\pm 1}$, respectively,
as outer and inner boundary labels,
where $u_s$ and $u_{s'}$ ($s, s' \in I_1(r) \cup I_2(r)$) are the cyclically reduced words in
$\{a,b\}$ obtained from the simple loops $\alpha_s$ and $\alpha_{s'}$,
respectively, as in Lemma~\ref{presentation}.
We obtain
the following very strong structure theorem.

\begin{theorem}{\rm (Structure Theorem)}
\label{structure}
Suppose that $r$ is a rational number with $0<r<1$.
Let $R$ be the symmetrized subset of $F(a, b)$ generated
by the single relator $u_{r}$ of the upper presentation of $G(K(r))$,
and let $S(r)= (S_1, S_2, S_1, S_2)$ be as in Proposition~\ref{sequence}.
Suppose that $M$ is a reduced nontrivial annular $R$-diagram such that,
for $\alpha$ and $\delta$ which are, respectively, arbitrary
outer and inner boundary cycles of $M$,
\begin{enumerate}[\indent \rm (i)]
\item the words $\phi(\alpha)$ and $\phi(\delta)$ are cyclically reduced;

\item the words $\phi(\alpha)$ and $\phi(\delta)$ are cyclically alternating;

\item
the cyclic $S$-sequences of the cyclic words $(\phi(\alpha))$ and $(\phi(\delta))$
do not contain $(S_1, S_2)$ nor $(S_2, S_1)$ as a subsequence.
\end{enumerate}
Let the outer and inner boundaries of $M$ be denoted by
$\sigma$ and $\tau$, respectively.
Then the following hold.
\begin{enumerate} [\indent \rm (1)]
\item The outer and inner boundaries $\sigma$ and $\tau$ are simple,
i.e., they are homeomorphic to the circle,
and there is no edge contained in $\sigma \cap \tau$.

\item $d_M(v)=2$ or $4$ for every vertex $v$ of $\partial M$.
Moreover, on both $\sigma$ and $\tau$, vertices of degree $2$
appear alternately with vertices of degree $4$.

\item $d_M(v)=4$ for every interior vertex $v$ of $M$.

\item $d_M(D)=4$ for every face $D$ of $M$.
\end{enumerate}
\end{theorem}

Before proving the theorem,
we prepare the following lemma.

\begin{lemma}
\label{degree2_vertices}
Under the assumption of Theorem~\ref{structure},
the following hold.
\begin{enumerate} [\indent \rm (1)]
\item There is no face $D$ in $M$ such that
$\partial D \cap \partial M$ contains three edges
$e_1$, $e_2$ and $e_3$ such that
$\bar{e}_1 \cap \bar{e}_2=\{v_1\}$ and $\bar{e}_2 \cap \bar{e}_3=\{v_2\}$,
where $d_M(v_i)=2$ for each $i=1, 2$.

\item The outer and inner boundaries $\sigma$ and $\tau$ are simple.
\end{enumerate}
\end{lemma}

\begin{proof}
(1) Suppose on the contrary that there is a face $D$ in $M$
and edges $e_1$, $e_2$ and $e_3$ in $\partial D \cap \partial M$
satisfying the condition.
Then $e_1$, $e_2$ and $e_3$ form three consecutive edges
in the outer or inner boundary cycle,
say the outer boundary cycle $\alpha$.
By Convention~\ref{convention}(2)--(3), the product $\phi(e_1) \phi(e_2) \phi(e_3)$
which is a subword of the cyclic word $(u_r^{\pm 1})$ cannot be expressed
as a product of less than $3$ pieces.
Because of the symmetry of $S_1$ and $S_2$ (Proposition~\ref{sequence}(1))
and Lemma~\ref{meaning_of_S-sequence+initial},
we may assume without loss of generality that
$\phi(e_1) \phi(e_2) \phi(e_3)$ is a subword of
the cyclic word $(u_r)$.
We also assume that the length $k$
of the continued fraction $r=[m_1,m_2, \dots,m_k]$ is greater than $1$.
(The proof for the case $k=1$ is analogous to the proof
for the general case $k\ge 2$.)
Let $w_0$ be the maximal $2$-piece which forms a proper initial
subword of $\phi(e_1) \phi(e_2) \phi(e_3)$.
Then $w_0$ is equal to one of the words in
Lemma~\ref{max_piece}(2b).
If $w_0$ is equal to $v_1 v_2$ or $v_{1e} v_2 v_{3b}^*$, then
$\phi(e_1) \phi(e_2) \phi(e_3)$ contains a subword $w$
such that the $S$-sequence of
$w$ is $(S_1,S_2,\ell)$ or
$(\ell, S_2, S_1)$ accordingly,
for some positive integer $\ell$.
By Proposition~\ref{sequence}, this implies that
the cyclic $S$-sequences of one of the cyclic words $(\phi(\alpha))$ and $(\phi(\delta))$
contains $(S_1, S_2)$ or $(S_2, S_1)$ as a subsequence.
This contradicts the hypothesis (iii).
The remaining possibilities for $w_0$ can be treated similarly.

(2) Suppose on the contrary that $\sigma$ or $\tau$ is not simple.
Then there is an extremal disk, say $J$, which is properly
contained in $M$ and connected to the rest of $M$ by a single vertex.
Here, recall that an {\it extremal disk} of a map $M$
is a submap of $M$ which is topologically a disk
and which has a boundary cycle $e_1, \dots, e_n$
such that the edges $e_1, \dots, e_n$ occur in order
in some boundary cycle of the whole map $M$.
Clearly $J$ is a connected and simply connected map
having at least one face.
Furthermore, by Corollary~\ref{small_cancellation_condition_2},
$J$ is a $[4, 4]$-map.
Then by Lemma~\ref{inequality},
we have
\[ \tag{\ddag} \label{curvature_formula}
\sum_{v \in \, \partial J} (3-d_J(v)) \ge 4.
\]
Putting
\[
A=\{v \in \partial J \, | \, d_J(v)=2\} \quad \text{and} \quad
B=\{v \in \partial J \, | \, d_J(v) \ge 4\},
\]
it is easy to see that $A$ has at least $4$ more elements than $B$ does
in order to satisfy inequality~(\ref{curvature_formula}).
Since $J$ is connected to the rest of $M$ by a single vertex,
say $v_0$, then every vertex in $\partial J$ except $v_0$ belongs to either $A$ or $B$
and $d_J(v_0)=d_M(v_0)-1 \ge 3$ (note that $d_M(v_0) \ge 4$, since $v_0 \in \partial M$).
So there are at least $2$ adjacent vertices,
say $v_1$ and $v_2$, belonging to $A$.
But then there is a face $D$ in $J$, so in $M$, such that
$\partial D \cap \partial M$ contains three consecutive edges
$e_1$, $e_2$ and $e_3$ such that
$\bar{e}_1 \cap \bar{e}_2=\{v_1\}$ and $\bar{e}_2 \cap \bar{e}_3=\{v_2\}$,
where $d_M(v_i)=d_J(v_i)=2$ for each $i=1, 2$,
contradicting Lemma~\ref{degree2_vertices}(1).
\end{proof}

\begin{proof}{\it of Theorem~\ref{structure} }
Note first that $M$ is a connected annular $[4, 4]$-map by
Corollary~\ref{small_cancellation_condition_2}.
By hypothesis (i), there is no vertex of degree $1$ in $\partial M$.
Moreover, there is no vertex of degree $3$ in $\partial M$,
as is shown in the following.
Suppose there is a vertex $v\in\partial M$ of degree $3$.
Then there are at most two faces which contain $v$.
If there are two such faces, then one of
$\phi(\alpha)$ and $\phi(\delta)$
is not cyclically alternating, a contradiction to hypothesis (ii).
If there is a unique such face, then by using hypothesis (ii),
we see that the boundary label
of the face is not cyclically alternating, a contradiction.
Finally, if there is no such face,
then this contradicts hypothesis (ii).
Hence, no vertex in $\partial M$ has degree $1$ nor $3$, and
so every vertex in $\partial M$ must have degree $2$ or at least $4$.

We argue two cases separately.

\medskip
\noindent {\bf Case 1.} $\sigma \cap \tau= \emptyset$.
\medskip

In this case, (1) follows immediately
from Lemma~\ref{degree2_vertices}(2).

(2) Since $\partial M$ is the disjoint union of
$\sigma$ and $\tau$, Lemma~\ref{inequality} yields
\[
0 \le \sum_{v \in \sigma} (3-d_M(v))+\sum_{v \in \tau} (3-d_M(v)).
\]
On the other hand, since $\sigma$ and $\tau$ are simple
by Lemma~\ref{degree2_vertices}(2)
and since they are disjoint by the current assumption,
every edge in $\sigma$ or $\tau$ is contained in
the boundary of a unique face of $M$.
Thus Lemma~\ref{degree2_vertices}(1) implies that
vertices of degree $2$ do not occur consecutively
on $\sigma$ nor on $\tau$.
Hence, the above inequality holds only when
$d_M(v)=2$ or $d_M(v) = 4$ for every vertex $v \in \sigma \cup \tau$
and when vertices of degree $2$ appear
alternately with vertices of degree $4$
on both $\sigma$ and $\tau$, thus proving (2).

(3)--(4) By (2),
$\sum_{v \in \partial M} (3-d_M(v))=0$.
This together with Lemma~\ref{inequality} yields
\[
0 \le \sum_{v \in M -\partial M} (4-d_M(v))+ \sum_{D \in M} (4-d_M(D)).
\]
Here, since $4-d_M(v) \le 0$ for every $v \in M -\partial M$ and
$4-d_M(D) \le 0$ for every $D \in M$ by the definition of a $[4,4]$-map,
the only possibility is that
$4-d_M(v) = 0$ for every vertex $v \in M -\partial M$ and
$4-d_M(D) = 0$ for every face $D \in M$, thus proving
(3) and (4).

\medskip
\noindent {\bf Case 2.} $\sigma \cap \tau \neq \emptyset$.
\medskip

(1) Suppose on the contrary that $\sigma \cap \tau$ contains an edge.
As illustrated in Figure~\ref{fig.island},
there is a submap $J$ of $M$ such that
\begin{enumerate}[\indent \rm (i)]
\item $J$ is bounded by a simple
closed path of the form $\sigma_1 \tau_1$, where $\sigma_1 \subseteq \sigma$
and $\tau_1 \subseteq \tau$;

\item $J$ is connected to the rest of
$M$ by two distinct vertices, say $v_1$ and $v_2$,
where $\sigma_1 \cap \tau_1=\{v_1, v_2 \}$
and $v_1$ is an endpoint of an edge contained in $\sigma \cap \tau$.
Note that $d_J(v_1)=d_M(v_1)-1 \ge 3$ and
$d_J(v_2)\ge 2$.
\end{enumerate}
Since $J$ is a connected and simply connected $[4,4]$-map,
Lemma~\ref{inequality} yields
\[
4 \le \sum_{v \in \partial J} (3-d_J(v)).
\]
Since $d_J(v_1) \ge 3$ and $d_J(v_2) \ge 2$,
this inequality implies
\[
3 \le \sum_{v \in \partial J-\{v_1,v_2\}} (3-d_J(v)).
\]
On the other hand, since every vertex in $\partial J -\{v_1, v_2\}$
has degree $2$ or at least $4$ and since
degree $2$ vertices cannot occur consecutively
on $\sigma_1-\{v_1, v_2\}$ nor on $\tau_1-\{v_1, v_2\}$
(see Lemma~\ref{degree2_vertices}(1)),
\[
\sum_{v \in \partial J-\{v_1,v_2\}} (3-d_J(v)) \le 2,
\]
a contradiction.

\begin{figure}[h]
\includegraphics{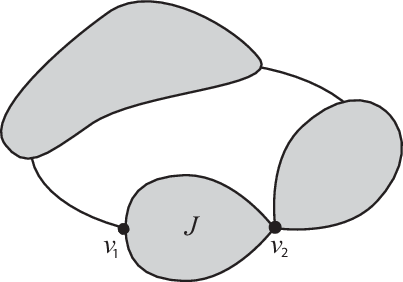}
\caption{
A possible annular map $M$ when $\sigma \cap \tau$ contains an edge.}
\label{fig.island}
\end{figure}

(2)--(4)
By (1), $\sigma \cap \tau$
consists of finitely many vertices in $M$.
Then, by Lemma~\ref{degree2_vertices}(1),
vertices of degree $2$ do not occur consecutively
on $\sigma$ nor on $\tau$.

First suppose that $\sigma \cap \tau$ consists of at least two vertices,
say $v_1, \dots, v_n$, where $n \ge 2$ and where these vertices are indexed
according as there is a submap $J_i$ of $M$ for every
$i=1, \dots, n$ such that
\begin{enumerate}[\indent \rm (i)]
\item $J_i$ is bounded by a simple
closed path of the form $\sigma_i \tau_i$, where $\sigma_i \subseteq \sigma$
and $\tau_i \subseteq \tau$;

\item $J_i$ is connected to the rest of
$M$ by two distinct vertices, say $v_i$ and $v_{i+1}$,
where $\sigma_i \cap \tau_i=\{v_i, v_{i+1} \}$ and where
$d_{J_i}(v_i), d_{J_i}(v_{i+1}) \ge 2$ and
$d_{J_i}(v_{i+1})+d_{J_{i+1}}(v_{i+1})=d_M(v_{i+1})$
(taking the indices modulo $n$).
\end{enumerate}
Then each $J_i$ is a connected and simply connected $[4, 4]$-map
such that
$M=J_1 \cup \cdots \cup J_{n}$.
Moreover
$\sigma=\sigma_1 \cup \cdots \cup \sigma_{n}$
and
$\tau=\tau_1 \cup \cdots \cup \tau_{n}$.
The same argument as for $(M', v_0', v_0'')$ above applies to each
$(J_i, v_i, v_{i+1})$ to prove the assertions.

Next suppose that $\sigma \cap \tau$ consists of a single vertex,
say $v_0$. Cut $M$ open at $v_0$ to get a connected and simply connected
$[4,4]$-map $M'$. In this process, the vertex $v_0$ is separated into two distinct vertices,
say $v_0'$ and $v_0''$, in $M'$ such that $d_{M'}(v_0'), d_{M'}(v_0'') \ge 2$
and $d_{M'}(v_0')+d_{M'}(v_0'')=d_M(v_0)$. Then $M'$ is bounded by a simple
closed path of the form $\sigma_0 \tau_0$,
where $\sigma_0 \cap \tau_0=\{v_0', v_0''\}$.
Again by Lemma~\ref{inequality},
\[
4 \le \sum_{v \in \partial M'} (3-d_{M'}(v)).
\]
Considering that vertices of degree $2$ do not occur consecutively
on $\sigma_0-\{v_0', v_0''\}$ nor on $\tau_0-\{v_0', v_0''\}$
(see Lemma~\ref{degree2_vertices}(1)),
we see that only the equality can hold, and that the equality holds
only when $d_{M'}(v)=2$ or $4$ for every
vertex $v \in \sigma_0 \cup \tau_0 -\{v_0', v_0''\}$, $d_{M'}(v_0')=2=d_{M'}(v_0'')$
and when vertices of degree $2$ appear alternately with vertices of degree $4$
starting and ending with vertices of degree $2$
on both $\sigma_0-\{v_0', v_0''\}$ and $\tau_0-\{v_0', v_0''\}$.
This implies that $d_M(v)=2$ or $4$ for every vertex $v \in \partial M-\{v_0\}$,
$d_{M}(v_0)=4$, and that vertices of degree $2$ appear alternately with vertices of degree $4$
on both $\sigma$ and $\tau$, thus proving (2).

Since $\sum_{v \in \partial M'} (3-d_{M'}(v))=4$, Lemma~\ref{inequality} yields
\[
0 \le \sum_{v \in {M'} -\partial {M'}} (4-d_{M'}(v))+ \sum_{D \in M'} (4-d_{M'}(D)).
\]
Here, by the definition of a $[4,4]$-map, $4-d_{M'}(v) \le 0$
for every vertex $v \in M' -\partial M'$ and
$4-d_{M'}(D) \le 0$ for every face $D \in M'$, we must have
$d_M(v)=d_{M'}(v) = 4$ for every vertex $v \in M' -\partial M'$ and
$d_M(D) =d_{M'}(D)= 4$ for every face $D \in M'$, thus proving
(3) and (4).
\end{proof}

Theorem~\ref{structure} enables us to identify all possible shapes of
the annular maps $M$.
To describe the result,
we define the {\it outer boundary layer} of
an annular map $M$ to be the submap of $M$
consisting of all faces $D$
such that the intersection of $\partial D$ with the outer boundary of $M$ contains an edge,
together with the edges and vertices contained in $\partial D$.

\begin{theorem}
\label{cor:structure}
Let $M$ be a reduced annular $[4,4]$-map
satisfying the assumptions of Theorem~\ref{structure}.
Then Figure~\ref{layer}{\rm (}a{\rm )} illustrates the only possible type of
the outer boundary layer of $M$,
while Figure~\ref{layer}{\rm (}b{\rm )} illustrates
the only possible type of whole $M$.
{\rm (}All layers have the same number of faces,
but the number may vary, as may the number of layers.{\rm )}
\end{theorem}

\begin{figure}[h]
\includegraphics{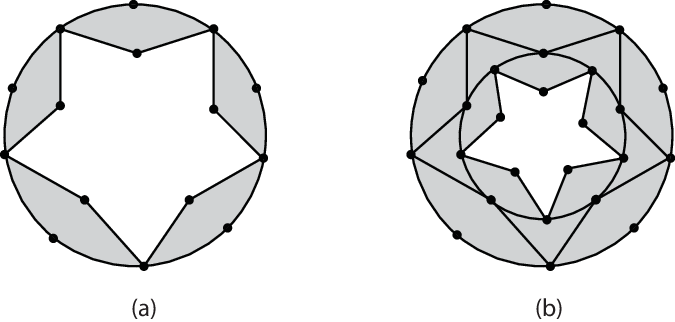}
\caption{}
\label{layer}
\end{figure}

\begin{proof}
Let $v_1, v_2, \dots, v_{2n}$ be the vertices
of the outer boundary $\sigma$ arranged in this cyclic order in $\sigma$,
such that $d_M(v_i)$ is $4$ or $2$
according to whether $i$ is odd or even
(see Theorem~\ref{structure}(1) and (2)).
Let $e_i$ ($1\le i \le 2n$) be the oriented edge in $\sigma$ running from $v_i$
to $v_{i+1}$,
where the indices are taken modulo $2n$.
Then, for each $j$ ($1\le j\le n$),
there is a unique face, $D_j$, of $M$
whose boundary contains the vertices
$v_{2j-1}$, $v_{2j}$ and $v_{2j+1}$
and the edges $e_{2j-1}$ and $e_{2j}$.
Since $d_M(D_j)=4$ by Theorem~\ref{structure}(4),
there are oriented edges $e_{2j-1}'$ and $e_{2j}'$ and of $M$ such that
$e_{2j-1}, e_{2j}, e_{2j}'^{-1}, e_{2j-1}'^{-1}$
is a boundary cycle of $D_j$.
The terminal point of $e_{2j-1}'$ is equal to the initial point of $e_{2j}'$,
and we denote it by $v_{2j}'$ (see Figure~\ref{2-cell}).

\begin{figure}[h]
\includegraphics{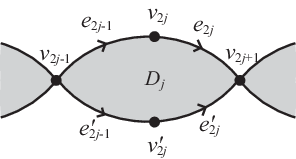}
\caption{\label{2-cell}}
\end{figure}

\medskip
\noindent{\bf Claim 1.}
{\it
The unoriented edges determined by
$e_i'$ and $e_{i+1}'$ are distinct for every $1\le i \le 2n$,
where indices are taken modulo $2n$.
}

\begin{proof}{\it of Claim~1 }
Suppose this is not the case.
Then, since $M$ is planar, we have $e_i'=e_{i+1}'^{-1}$ for some $i$.
Suppose first that $e_{2j}'=e_{2j+1}'^{-1}$ for some $j$.
Then the union of the closures $\bar D_j$ and $\bar D_{j+1}$ forms
a neighborhood of $v_{2j+1}$ in $M$ and hence,
$e_{2j}, e_{2j+1}^{-1}$ and
$e_{2j}'=e_{2j+1}'^{-1}$ are the only oriented edges
of $M$ having $v_{2j+1}$ as the terminal point.
So $d_M(v_{2j+1})=3$, a contradiction to Theorem~\ref{structure}(2).
Next, suppose $e_{2j-1}'=e_{2j}'^{-1}$ for some $j$.
Then $v_{2j}'$ is contained in the interior of the closure $\bar D_j$, and
$e_{2j-1}'=e_{2j}'^{-1}$ is the only
oriented edge of $M$ having $v_{2j}'$ as the terminal point.
So $d_M(v_{2j}')=1$, a contradiction to Theorem~\ref{structure}(3).
\end{proof}

\medskip
\noindent{\bf Claim 2.}
{\it
For every $1\le j \le n$, the oriented edges
$e_{2j}, e_{2j+1}^{-1}, e_{2j}',e_{2j+1}'^{-1}$
are mutually distinct, and they are the only oriented edges of $M$
which have $v_{2j+1}$ as the terminal point.
}

\begin{proof}{\it of Claim~2 }
Let $U$ be a small regular neighborhood of $v_{2j+1}$ relative to $M$,
namely, $U$ is a closed disk in $\RR^2$
containing $v_{2j+1}$ in its interior,
such that $|M^{(1)}|\cap U$ consists of $4$ ($=d_M(v_{2j+1})$)
arcs joining $v_{2j+1}$ with $\partial U$,
each pair of which intersect only at $v_{2j+1}$.
Here $|M^{(1)}|$ denotes the $1$-skeleton of $M$,
i.e., the union of the vertices and edges of $M$.
Consider the punctured disk $\check U:=U-\{v_{2j+1}\}$.
Then the subsets $(e_{2j}\cup e_{2j+1})\cap \check U$
and $(e_{2j}'\cup e_{2j+1}')\cap \check U$ of $\check U$ are separated
in $\check U$ by the union of open $2$-cells $D_j\cup D_{j+1}$.
Hence each of the oriented edges $e_{2j}'$ and $e_{2j+1}'^{-1}$ cannot be identical with
any of the oriented edges $e_{2j}$ and $e_{2j+1}^{-1}$.
Since $e_{2j}\ne e_{2j+1}^{-1}$ by Theorem~\ref{structure}(1) and
$e_{2j}'\ne e_{2j+1}'^{-1}$ by Claim~1,
the oriented edges
$e_{2j}, e_{2j+1}^{-1}, e_{2j}',e_{2j+1}'^{-1}$
are mutually distinct.
Since $d_M(v_{2j+1})=4$, we obtain the claim.
\end{proof}

\medskip
\noindent{\bf Claim 3.}
{\it
$e_i'$ is not contained in $\sigma$ for every $1\le i \le 2n$.
}

\begin{proof}{\it of Claim~3 }
We prove the claim when $i=2j$ for some integer $j$.
(The remaining case can be treated by a parallel argument.)
Suppose $e_{2j}'$ is contained in $\sigma$.
Then, since $v_{2j+1}$ is the terminal point of $e_{2j}'$,
the oriented edge $e_{2j}'$ is equal to $e_{2j}$
or $e_{2j+1}^{-1}$.
But, this is impossible by Claim~2.
\end{proof}

\medskip
\noindent{\bf Claim 4.}
{\it
$v_{2j}'$ is not contained in $\sigma$ for every $1\le j \le n$.
}

\begin{proof}{\it of Claim~4 }
Suppose $v_{2j}'$ is contained in $\sigma$ for some $j$.
Then $v_{2j}'=v_k$ for some $1 \le k \le 2n$,
and the oriented edges $e_{k-1}, e_k^{-1}, e_{2j-1}',e_{2j}'^{-1}$
have $v_{2j}'=v_k$ as the terminal point.
Since these are mutually distinct by Claims~1 and 3,
we have $d_M(v_k)\ge 4$ and hence $k=2h+1$ for some integer $h$.
If $h=j$, then $e_{2j}'$ is a loop based on $v_{2j}'=v_{2j+1}$
and hence the oriented edge $e_{2j}'^{-1}$ also
has $v_{2j+1}$ as the terminal point,
a contradiction to Claim~2.
Similarly, we see $h=j-1$ cannot happen.
Hence we have $v_{2j}'=v_{2h+1}$,
where $h \neq j, j-1$.

Note that in addition to
$e_{2h}, e_{2h+1}^{-1}, e_{2h}', e_{2h+1}'^{-1}$,
the oriented edges $e_{2j-1}'$ and $e_{2j}'^{-1}$
have $v_{2j}'=v_{2h+1}$ as the terminal point.
By Claims~1 and 3,
this implies that
the pair $(e_{2j-1}', e_{2j}'^{-1})$ is equal to
$(e_{2h}', e_{2h+1}'^{-1})$ or $(e_{2h+1}'^{-1},e_{2h}')$.
Since $M$ is planar, $(e_{2j-1}', e_{2j}'^{-1})$ cannot be equal to
$(e_{2h}', e_{2h+1}'^{-1})$.
So we may assume
$(e_{2j-1}', e_{2j}'^{-1})=(e_{2h+1}'^{-1}, e_{2h}')$.
Then the initial point $v_{2j+1}$ of $e_{2j}'^{-1}$
is equal to the initial point $v_{2h}'$ of $e_{2h}'$.
Similarly,
we have $v_{2j-1}=v_{2h+2}'$.
Thus we have shown that the identity
$v_{2j}'=v_{2h+1}$
implies the identities
$v_{2h}'=v_{2j+1}$ and $v_{2h+2}'=v_{2j-1}$.
By repeatedly applying this fact,
we see that,
for every integer $k$,
$v_{2(h-k)}'=v_{2(j+k)+1}$ and $v_{2(j+k)}'=v_{2(h-k)+1}$.
Thus we can find a pair of integers $j^*$ and $h^*$
such that $v_{2j^*}'=v_{2h^*+1}$ and $h^*=j^*$ or $j^*-1$.
However, this is impossible by the argument
in the first paragraph of this proof.
\end{proof}

\medskip
\noindent{\bf Claim 5.}
{\it
The unoriented edges determined by
$e_i'$ $(1\le i \le 2n)$ are mutually distinct.
}

\begin{proof}{\it of Claim~5 }
Suppose this is not the case.
Then, since $M$ is planar,
we have $e_i'=e_k'^{-1}$ for some $1 \le i < k \le 2n$.
We assume $i=2j$ for some integer $j$.
(The other case is treated by parallel arguments.)
If $k=2h$ for some integer $h$,
then $e_{2j}'=e_{2h}'^{-1}$, and hence
the initial point $v_{2j}'$ of $e_{2j}'$
is equal to the initial point $v_{2h+1}$ of $e_{2h}'^{-1}$.
This contradicts Claim~4.
If $k=2h+1$ for some integer $h$,
then $e_{2j}'=e_{2h+1}'^{-1}$, and hence
the terminal point $v_{2j+1}$ of $e_{2j}'$
is equal to the terminal point $v_{2h+1}$ of $e_{2h+1}'^{-1}$.
Since $\sigma$ is simple, this implies $j=h$
and hence $e_{2j}'=e_{2j+1}'^{-1}$.
This contradicts Claim~1.
\end{proof}

\medskip
\noindent{\bf Claim 6.}
{\it
The vertices $v_{2j}'$ $(1\le j\le n)$ are mutually distinct.
}

\begin{proof}{\it of Claim~6 }
Suppose $v_{2j}'= v_{2h}'$ for some $1 \le j \neq h \le n$.
Then $e_{2j-1}'$, $e_{2j}'^{-1}$, $e_{2h-1}'$, $e_{2h}'^{-1}$
have $v_{2j}'= v_{2h}'$ as the terminal point.
Since these oriented edges are mutually distinct by Claim~5
and since $d_M(v_{2j}')=2$ or $4$ by Theorem~\ref{structure}(2),(3),
we see that $d_M(v_{2j}')=4$
and that these are the only oriented edges having
$v_{2j}'= v_{2h}'$ as the terminal point.
We can choose $j$ and $h$,
so that they are {\it outermost} in the sense that
$1 \le j < h \le n$ (after a cyclic permutation of indices) and
that the vertices $v_{2j}'= v_{2h}'$ and $v_{2k}'$ ($j<k<h$)
are mutually distinct.
Then $\cup_{2j \le i \le 2h-1} \bar{e}_i'$ is a simple loop.
We show that $e_{2j}'$, $e_{2h-1}'$ are contained in
the boundary of the annular diagram $M$.
Suppose this is not the case.
Then, by the above observations,
there is a face, $D$, of $M$ whose
boundary contains $e_{2j}'$, $e_{2h-1}'$.
By Claim~2, the edges $e_{2j+1}'$ and $e_{2h-2}'$
must be contained in $\partial D$.
Since $d_M(D)=4$ by Theorem~\ref{structure}(4),
a boundary cycle of $D$ is given by
$e_{2j}'$, $e_{2j+1}'$, $e_{2h-2}'$, $e_{2h-1}'$.
So, $h=j+2$ and $v_{2j+2}'$ is
contained in the interior of $\bar D \cup \bar D_{j+1}$.
This implies that $v_{2j+2}'$ is an inner vertex and
$d_M(v_{2j+2}')=2$, a contradiction to Theorem~\ref{structure}(3).
Hence $e_{2j}', e_{2h-1}'$ are contained in
the boundary of the annular diagram $M$,
and in particular, $v_{2h-1}, v_{2h}'=v_{2j}', v_{2j+1}$
lie in the inner boundary of $M$ successively.
However all of them have degree $4$.
This contradicts Theorem~\ref{structure}(2).
\end{proof}

By Claims~5 and 6, we see that
$\sigma':=\cup_{i=1}^{2n} \bar{e}_i'$ is a simple loop and
that $\sigma\cap\sigma'=\{v_{2j+1} \ | \ 0\le j\le n-1\}$.
Thus the outer boundary layer, $J$, of $M$
is as illustrated in Figure~\ref{layer}(a).
If $J=M$, we obtain the conclusion of Theorem~\ref{cor:structure}.
So, in the following, we assume $J\subsetneqq M$.
Let $M'$ be the map obtained from $M$
by collapsing the outer boundary layer $J$ onto $\sigma'$.
Then $M'$ is an annular map
whose outer and inner boundaries are
the simple loops $\sigma'$ and $\tau$, respectively.

\medskip
\noindent{\bf Claim 7.}
{\it
The annular map $M'$ satisfies the conclusions
of Theorem~\ref{structure}.
Moreover, for each $1\le j\le n$,
there is a face $D'_j$ of $M'$
such that $e_{2j}'$ and $e'_{2j+1}$ are the only edges of $\sigma'$
contained in $\partial D_j'$.
}

\begin{proof}{\it of Claim~7 }
Since $J\subsetneqq M$,
we may assume, after cyclic permutation of indices, that
there is a face $D_1'$ of $M'$
whose boundary contains either $e_2'$ or $e_3'$.
Then $\partial D_1'$ contains both $e_2'$ and $e_3'$,
because $d_M(v_3)=4$.
By using the fact that $d_M(v_1)=d_M(v_5)=4$,
we see that $e_2'$ and $e_3'$ are the only edges of $\sigma'$
contained in $\partial D_1'$.
Since $d_M(D_1')=4$,
there are oriented edges $e_2''$ and $e_3''$ of $M$
such that a boundary cycle of $D_1'$ is
$e_2', e_3', e_3''^{-1}, e_2''^{-1}$.
This implies that $d_M(v_4')\ge 3$ and so $d_M(v_4')= 4$.
Thus there is an oriented edge, $e_4''$, emanating from $v_4'$
other than $e_3'^{-1}, e_3''^{-1}, e_4'$.
Since $d_M(v_4')=d_M(v_5)=4$,
we see by Theorem~\ref{structure}(2) that
the edge $e_4'$ is not contained in the boundary of $M$.
Hence, there is a face, $D_2'$, of $M$
whose boundary contains both $e_4'$ and $e_4''$.
We see, by repeating the preceding argument for $D_1'$,
that $e_5'$ is also contained in $\partial D_2'$
and that $e_4'$ and $e_5'$ are
the only edges of $\sigma'$
contained in $\partial D_2'$.
Here, observe that $D_1'=D_2'$ if $n=1$.
By repeating this argument,
we see that, for each $1\le j\le n$,
there is a face $D'_j$ of $M'$
such that $e_{2j}'$ and $e'_{2j+1}$ are the only edges of $\sigma'$
contained in $\partial D_j'$.
Hence the outer boundary $\sigma'$
and the inner boundary $\tau$ of $M'$ are disjoint.
So $M'$ satisfies the conclusion (1) of Theorem~\ref{structure}.
The above fact also implies
$d_{M'}(v_{2j}')=4$ and $d_{M'}(v_{2j+1})=2$ for every $j$.
Thus the outer boundary $\sigma'$ of $M'$ satisfies the
conclusion (2).
Since $d_{M'}(v)=d_M(v)$ for any vertex of $M'$ which is not contained in $\sigma$
and since $\tau$ is disjoint from $\sigma$,
the inner boundary $\tau$ of $M'$ also satisfies the conclusion (2)
and $M'$ satisfies the conclusion (3).
It is obvious that $M'$ satisfies the conclusion (4).
\end{proof}

Note that the arguments preceding Claim~7
use only the conclusion of Theorem~\ref{structure}.
Thus we can repeat the arguments and see that
the outer boundary layer, $J'$, of $M'$ also satisfies the conclusion
of Theorem~\ref{cor:structure}.
If $J\cup J'=M$, then we see from Claim~7 that
$M$ is as described in Theorem~\ref{cor:structure}.
If $J\cup J'\subsetneqq M$, we can repeat these arguments,
and obtain Theorem~\ref{cor:structure}.
\end{proof}

\section{Proof of the only if part of Main Theorem~\ref{main_theorem}}
\label{sec:only_if_part}

We first note that the special case where $p=2$ can be settled very easily
as follows. In this case,
$K(1/2)$ is the Hopf link, and
$G(K(1/2))=\langle a,b \ | \ u_{1/2} \rangle$
with $u_{1/2}=aba^{-1}b^{-1}$ is the rank $2$ free abelian group
with basis $\{a, b\}$.
On the other hand, $I_1(1/2) \cup I_2(1/2)=\{0,1\}$, and
it is obvious that $u_0=ab$ is not conjugate to $u_1^{\pm 1}=(ab^{-1})^{\pm 1}$
in $G(K(1/2))$, so that $\alpha_0$ is not homotopic to $\alpha_1$ in $S^3-K(1/2)$.
This implies that the unoriented simple loops $\alpha_s$
fall into precisely three homotopy classes
(represented by $\alpha_0$, $\alpha_1$ and $\alpha_{\infty}$, the last one trivial).
Hence the only if part of Main Theorem~\ref{main_theorem} is valid
when $p=2$.

\begin{proof}{\it of the only if part of Main Theorem~\ref{main_theorem} }
As noted in the beginning of this section,
we may assume that $p$ is a positive integer with $p \ge 3$.
For two distinct elements $s, s' \in I_1(1/p) \cup I_2(1/p)=\{0\} \cup [{1 \over {p-1}}, 1]$,
suppose that the unoriented loops
$\alpha_s$ and $\alpha_{s'}$ are homotopic in $S^3-K(1/p)$.
Then $u_s$ and $u_{s'}^{\pm 1}$ are conjugate in $G(K(1/p))$.
Let $R$ be the symmetrized subset of $F(a, b)$ generated
by the single relator $u_{1/p}$ of the upper presentation of $G(K(1/p))$.
Due to Lemma~\ref{lyndon_schupp},
there is a reduced nontrivial annular $R$-diagram $M$ such that
$u_s$ and $u_{s'}^{\pm 1}$ are, respectively, outer and inner
boundary labels of $M$.
Then we see that $M$ satisfies
the three hypotheses (i), (ii) and (iii) of Theorem~\ref{structure}.
In fact, hypothesis (iii) follows from
Proposition~\ref{connection}(1)
and Remark~\ref{rem:extreme_case}.

\begin{lemma}
\label{Outer/inner_cycles}
Suppose that $\lp b_1, b_2, \dots, b_n \rp$ denotes
the cyclic $S$-sequence of an outer boundary label of $M$.
Then $b_i<p$ and the cyclic $S$-sequence of an inner boundary label of $M$
is $\lp p-b_1, p-b_2, \dots, p-b_n \rp$.
\end{lemma}

We complete the proof of
the only if part of Main Theorem~\ref{main_theorem},
by assuming the above lemma.
Note that the cyclic $S$-sequences of arbitrary outer
and inner boundary labels of $M$
are
$CS(u_s)$ and $CS(u_{s'}^{\pm 1})=CS(u_{s'})$,
respectively.
Putting $CS(u_s)=\lp b_1, b_2, \dots, b_n \rp$,
Lemma~\ref{Outer/inner_cycles} implies
that $CS(u_{s'})=\lp p-b_1, p-b_2, \dots, p-b_n \rp$.
In particular, $CS(u_s)$ and $CS(u_{s'})$ have the same length.
Since the length of $CS(u_s)$ is even or $1$
according as $s \neq 0$ or $s=0$
(see Lemma~\ref{j-term}(1) and Remark~\ref{remark:j-term})
and since $s$ and $s'$ are distinct,
we see that $s$ nor $s'$ is $0$.
Hence we may
write $s=q_1/p_1$ and $s'=q_2/p_2$, where $(p_i, q_i)$ is a pair of
relatively prime positive integers.
Then by
Lemma~\ref{j-term}(1),
we have $n=2q_1=2q_2$, so $q_1=q_2$.
Also by
Lemma~\ref{j-term}(1),
\[
\sum_{i=1}^{q_1} b_i=\sum_{i=1}^{q_1} \nu_i(s)
= \lfloor p_1 \rfloor_* - \lfloor 0 \rfloor_*=p_1,
\]
where $\nu_i(s)$ denotes the $i$-th term of $S(u_s)$,
and similarly
\[
\sum_{i=1}^{q_1} (p-b_i)=\sum_{i=1}^{q_1} \nu_i(s')
=\lfloor p_2 \rfloor_* - \lfloor 0 \rfloor_*
=p_2,
\]
where $\nu_i(s')$ denotes the $i$-th term of $S(u_{s'})$.
Hence $p_2=\sum_{i=1}^{q_1} (p-b_i)=pq_1-p_1$,
which implies that $q_1/(p_1+p_2)=1/p$, as required.
\end{proof}

It remains to prove Lemma~\ref{Outer/inner_cycles}.

\begin{proof}{\it of Lemma~\ref{Outer/inner_cycles} }
Let $J$ be the outer boundary layer of $M$.
Due to
Theorem~\ref{cor:structure},
$J$ is as depicted in Figure~\ref{layer}(a).
Let $\alpha$ and $\delta$ be, respectively, the outer and
inner boundary cycles of $J$ starting from $v_0$,
where $v_0$ is a vertex lying in both
the outer and inner boundaries of $J$.
Here, recall from Convention~\ref{convention2} that
$\alpha$ is read clockwise and $\delta$ is read counterclockwise.
Let $\alpha=e_1, e_2, \dots, e_{2m}$ and
$\delta^{-1}=e_1', e_2', \dots, e_{2m}'$ be the
decompositions into oriented edges in $\partial J$.
Then clearly
for each $j=1, \dots, m$,
there is a face $D_j$ of $J$ such that
$e_{2j-1}, e_{2j}, e_{2j}'^{-1}, e_{2j-1}'^{-1}$ are
consecutive edges in a boundary cycle of $D_j$.

\medskip
\noindent{\bf Claim 1.}
{\it
For each $2$-cell $D_j$,
there are decompositions
$v_2\equiv v_{2b}v_{2e}$ and $v_4\equiv v_{4b}v_{4e}$
such that the following hold,
where $u_{1/p} \equiv v_2v_4$ as in Lemma~\ref{max_piece},
and $v_{ib}$ and $v_{ie}$ are nonempty initial and terminal subwords of $v_i$.
\begin{enumerate}[\indent \rm (1)]
\item
If $(\phi(\partial D_j)) \equiv (u_{1/p})$, then
$(\phi(e_{2j-1}), \phi(e_{2j}), \phi(e_{2j}'^{-1}), \phi(e_{2j-1}'^{-1}))$
is equal to
$(v_{2e},v_{4b},v_{4e},v_{2b})$ or
$(v_{4e},v_{2b},v_{2e},v_{4b})$.
\item
If $(\phi(\partial D_j)) \equiv (u_{1/p}^{-1})$, then
$(\phi(e_{2j}^{-1}), \phi(e_{2j-1}^{-1}), \phi(e_{2j-1}'), \phi(e_{2j}'))$
is equal to
$(v_{2e},v_{4b},v_{4e},v_{2b})$ or
$(v_{4e},v_{2b},v_{2e},v_{4b})$.
\end{enumerate}
}

\begin{proof}{\it of Claim~1 }
We treat the case where $(\phi(\partial D_j)) \equiv (u_{1/p})$.
(The other case can be treated similarly.)
Since $\phi(e_{2j-1})\phi(e_{2j})$ is not a piece by
Convention~\ref{convention},
it contains a maximal $1$-piece as a proper initial subword, $v$.
Assume that $v$ is $v_{2b}^*$ (resp., $v_{4b}^*$)
(see Lemma~\ref{max_piece}(1a)).
Since $\phi(e_{2j-1})\phi(e_{2j})$ is contained in a
maximal $2$-piece,
this assumption together with Lemma~\ref{max_piece}(1b)
implies that $\phi(e_{2j-1})\phi(e_{2j})$ is $v_2$ or $v_4$, and hence
$S(\phi(e_{2j-1})\phi(e_{2j}))=(p)$.
But, since $s \in I_1(1/p) \cup I_2(1/p)=\{0\} \cup [{1 \over {p-1}}, 1]$,
every term of
the cyclic word $(u_s)=(\phi(\alpha))$ is at most $p-1$
by Proposition~\ref{connection}(2) and by Remark~\ref{rem:extreme_case}.
This implies that
the cyclic word $(u_s)=(\phi(\alpha))$ cannot contain a subword
whose $S$-sequence is $(p)$, a contradiction.
Hence,
the word $v$ is equal to $v_{2e}$ or $v_{4e}$
by Lemma~\ref{max_piece}(1a).
This together with Lemma~\ref{max_piece} implies that
$(\phi(e_{2j-1}),\phi(e_{2j}))=(v_{2e},v_{4b})$ or
$(v_{4e},v_{2b})$ accordingly.
By applying a parallel argument to
$(\phi(e_{2j}'^{-1}), \phi(e_{2j-1}'^{-1}))$ and
by using the fact that
the cyclic word $(\phi(\partial D_j))$ is equal to $(u_{1/p})$,
we obtain the desired result.
\end{proof}

By Claim~1, there are positive integers $\ell_i$ and $\ell_i'$,
for each $1\le i\le 2m$,
such that $S(\phi(e_i))=(\ell_i)$ and $S(\phi(e_i'))=(\ell_i')$
and such that $\ell_i+\ell_i'=p$. Furthermore, they satisfy the following.

\medskip
\noindent{\bf Claim 2.}
{\it
For each $1\le j\le m$, the following hold.
\begin{enumerate}[\indent \rm (1)]
\item
$S(\phi(e_{2j-1})\phi(e_{2j}))=(\ell_{2j-1},\ell_{2j})$
and
$S(\phi(e_{2j-1}')\phi(e_{2j}'))=(\ell_{2j-1}',\ell_{2j}')$.
\item
$S(\phi(e_{2j})\phi(e_{2j+1}))=(\ell_{2j},\ell_{2j+1})$
and
$S(\phi(e_{2j}')\phi(e_{2j+1}'))=(\ell_{2j}',\ell_{2j+1}')$,
where the indices are taken modulo $2m$.
\end{enumerate}
}

\begin{proof}{\it of Claim~2 }
The first assertion immediately follows from Claim~1.
To see the second assertion,
note that Claim~1 implies that
if $\phi(e_{2j-1})$ is a positive (resp., negative) word
then $\phi(e_{2j}')$ is also a positive (resp., negative) word
and $\phi(e_{2j-1}')$ and $\phi(e_{2j})$ are
negative (resp., positive) words.
In particular, if $\phi(e_i)$ is a positive (resp., negative) word
then $\phi(e_i')$ is a negative (resp., positive) word,
i.e., $\phi(e_i)$ and $\phi(e_i')$ always have ``opposite signs''.
Now consider the $S$-sequence of $\phi(e_{2j})\phi(e_{2j+1})$.
Clearly $S(\phi(e_{2j})\phi(e_{2j+1}))$
is either $(\ell_{2j}+\ell_{2j+1})$ or $(\ell_{2j}, \ell_{2j+1})$.
We will show that only the latter is possible.

Suppose on the contrary that $S(\phi(e_{2j})\phi(e_{2j+1}))$
is $(\ell_{2j}+\ell_{2j+1})$ for some $j$.
Then $\phi(e_{2j})$ and $\phi(e_{2j+1})$ share the same sign
and hence $\phi(e_{2j}')$ and $\phi(e_{2j+1}')$ share the same sign
by the above observation.
So $S(\phi(e_{2j}')\phi(e_{2j+1}'))$ is $(\ell_{2j}'+\ell_{2j+1}')$.
By Proposition~\ref{connection}(2)
and Remark~\ref{rem:extreme_case},
$(u_s)=(\phi(\alpha))$ cannot contain a subword
whose $S$-sequence is $(p)$. So $\ell_{2j}+\ell_{2j+1} < p$.
But then $S(\phi(e_{2j}')\phi(e_{2j+1}'))$ is
$(\ell_{2j}'+\ell_{2j+1}')=(2p-\ell_{2j}-\ell_{2j+1})$ with
$2p-\ell_{2j}-\ell_{2j+1} > p$.
If $J=M$, this contradicts
the assumption that $s' \in I_1(1/p) \cup I_2(1/p)=\{0\} \cup [{1 \over {p-1}}, 1]$,
because every term of $CS(s')=CS(u_{s'})$ has to be at most $p-1$
again by Proposition~\ref{connection}(2)
and Remark~\ref{rem:extreme_case}.
Also if $J \subsetneq M$, then, as depicted in Figure~\ref{layer}(b),
$e_{2j}'$ and $e_{2j+1}'$ are two consecutive edges in
$\partial D_j' \cap \delta^{-1}$
(recall that $\delta$ is the inner boundary cycle of $J$ starting from $v_0$ and is read counterclockwise)
for some face $D_j'$ in $M-J$, but then $(u_{1/p})$ contains a subword whose $S$-sequence is
$(2p-\ell_{2j}-\ell_{2j+1})$ with $2p-\ell_{2j}-\ell_{2j+1} > p$, contradicting
$CS(u_{1/p})=\lp p,p \rp$.
\end{proof}

Therefore $CS(\phi(\alpha))$ becomes
$\lp \ell_{1}, \ell_{2}, \dots, \ell_{2m-1}, \ell_{2m} \rp$,
and $CS(\phi(\delta^{-1}))$ becomes
$\lp p-\ell_{1}, p-\ell_{2}, \dots, p-\ell_{2m-1}, p-\ell_{2m} \rp$.
If $M=J$, the proof of Lemma~\ref{Outer/inner_cycles} is completed.
Now assume $J \subsetneq M$. Let $J_1$ denote the outer boundary layer of
$M-int (J)$. Then the cyclic $S$-sequence of an arbitrary inner boundary label
of $J_1$ is $\lp \ell_{1}, \ell_{2}, \dots, \ell_{2m-1}, \ell_{2m} \rp$.
If $M=J \cup J_1$, then $s=s'$, a contradiction.
Hence we must have $J \cup J_1 \subsetneq M$.
Let $J_2$ denote the outer boundary layer of $M-int(J \cup J_1)$.
Then the cyclic $S$-sequence of an arbitrary inner boundary label
of $J_2$ is $\lp p-\ell_{1}, p-\ell_{2}, \dots, p-\ell_{2m-1}, p-\ell_{2m} \rp$.
If $M=J \cup J_1 \cup J_2$, the proof of
Lemma~\ref{Outer/inner_cycles} is completed.
By the repetition of this argument, we are done.
\end{proof}

\section*{Acknowledgement}
The authors would like to thank Koji Fujiwara
for his interest in this work and for inspiring discussion.
They would also like to thank the referee for very careful reading
and valuable comments and suggestions,
especially on the proof of the if part of the main theorem.
Their thanks also go to the editor of this journal
for kindly and properly handling this paper and its two sequels.

\bibstyle{plain}

\end{document}